\documentclass{article}
\usepackage{graphicx} 

\usepackage{graphicx}
\usepackage[utf8]{inputenc} 
\usepackage[T1]{fontenc}    
\usepackage{hyperref}      
\usepackage{url}            
\usepackage{booktabs}      
\usepackage{amsfonts}      
\usepackage{nicefrac}      
\usepackage{microtype}    
\usepackage{xcolor}        
\usepackage{amsmath}
\usepackage{amsthm}
\usepackage{amssymb} 
\usepackage{algorithm}
\usepackage{algorithmic}
\usepackage{mathtools} 
\usepackage{svg}
\usepackage{comment}
\usepackage{enumitem}

\DeclareMathOperator*{\argmax}{arg\,max}
\DeclareMathOperator*{\argmin}{arg\,min}

\DeclareMathOperator{\R}{\mathbb{R}}

\DeclareMathOperator{\setX}{\mathcal{X}}
\DeclareMathOperator{\I}{ \mathcal{I}} 
\DeclareMathOperator{\dom}{dom}
\DeclareMathOperator{\dist}{dist}

\newtheorem{theorem}{Theorem}[section]
\newtheorem{proposition}[theorem]{Proposition}

\newtheorem{lemma}[theorem]{Lemma}

\theoremstyle{definition}
\newtheorem{definition}[theorem]{Definition}

\newtheorem{assumption}[theorem]{Assumption}

\newcommand{\NN}{{\mathbb N}} 
\newcommand{\RR}{{\mathbb R}}

\title{
Strongly Convex Maximization via the Frank-Wolfe Algorithm with the Kurdyka-\L{}ojasiewicz Inequality}
\author{Fatih S. Akta\c{s}      \and
        Christian Kroer}
\date{}

\begin{document}

\maketitle

\begin{abstract}
We study the convergence properties of the ``greedy'' Frank-Wolfe algorithm with a unit step size, for a convex maximization problem over a compact set. We assume the function satisfies smoothness and strong convexity. These assumptions together with the Kurdyka-\L{}ojasiewicz (KL) property allow us to derive global asymptotic convergence for the sequence generated by the algorithm. Furthermore, we also derive a convergence rate that depends on the geometric properties of the problem. To illustrate the implications of the convergence result obtained, we prove a new convergence result for a sparse principal component analysis algorithm, propose a convergent reweighted $\ell_1$ minimization algorithm for compressed sensing, and design a new algorithm for the semidefinite relaxation of the Max-Cut problem.


\end{abstract}

\section{Introduction}

In this paper, we study the following convex \emph{maximization} model
\begin{equation}
\label{eq:optimization_model}
\begin{array}{ll@{}ll}
\underset{x}\max \quad & g(x)  \\
\text{st.} \quad & x \in \setX, \\
\end{array}
\end{equation}

where the following assumptions are made for this model. 

\begin{assumption}
\label{as:convex_smooth_compact}

Our blanket assumptions are:

\begin{enumerate}

    \item The function $g:\RR^n \mapsto \RR $ is strongly convex and smooth i.e. $\nabla g(x)$ is Lipschitz continuous.
    
    \item The constraint set $\setX \subseteq \R^n$ is a nonempty compact set.
    
\end{enumerate}

\end{assumption}

Notice that we do not assume the convexity of the feasible set $\setX$.

Applications of convex maximization problems are ubiquitous. For example, in optimization models where the objective satisfies economies of scale, the problem becomes a convex maximization problem, and similarly, many binary linear optimization problems can be equivalently modeled as a convex maximization problem (in particular with a quadratic objective) \cite{zwart1974global}. Many machine learning problems can also be formulated as a convex maximization problem. In \cite{mangasarian1996machine}, misclassification minimization and feature selection problems are formulated as convex maximization problems. The classical statistical analysis technique principal component analysis (PCA) can be formulated as a convex quadratic maximization problem. Extensions of the PCA algorithm (without using semidefinite relaxations) are also convex maximization problems, such as sparse PCA \cite{d2004direct}, nonnegative PCA, and nonnegative sparse PCA \cite{zass2006nonnegative}. Furthermore, reweighted $\ell_1$ norm type algorithms used in compressed sensing problems \cite{candes2008enhancing}, low-rank matrix recovery \cite{mohan2010reweighted}, and sparse PCA \cite{aktacs2023pca} are derived from convex maximization problems. Convex maximization problems also naturally arise in graph theory problems, such as the famous Max-Cut problem \cite{goemans1995improved}. In Section \ref{sec:application}, we show that even the computational models for the SDP relaxation of the Max-Cut problem can be represented as a convex maximization problem. In \cite{benson1995concave} and references therein, applications of convex maximization problems are discussed in detail, and many examples of integer linear and integer quadratic programs are shown to be equivalent to convex maximization. Another area where convex maximization naturally emerges is robust optimization when computing the worst-case scenario for a constraint that is convex in the uncertain parameter \cite{selvi2022convex,ben2022algorithm}. Additionally, as discussed in \cite{lipp2016variations}, in the difference of convex functions (DC) framework, early approaches reformulate the problem as a convex maximization. Finally, the convex-concave procedure (CCP) is shown to be a special case of the Frank-Wolfe (FW) algorithm applied to a convex maximization problem in \cite{yurtsever2022cccp}.

Maximizing a convex function is NP-hard for simple models such as maximizing a quadratic function over a hypercube, and even checking local optimality is NP-hard \cite{pardalos1988checking}. Thus, early approaches have been mostly based on linear approximations as overviewed in \cite{pardalos1986methods}. The main drawback of these approaches is that the subproblem cost at each iteration grows, making it computationally inefficient in practice. In \cite{audet2005essays,andrianova2016one}, an overview of all methods to solve convex quadratic maximization problems such as cutting plane, numerical approaches, and decomposition of the feasible set are given. A more recent approach given in \cite{selvi2022convex} adapts methods from robust optimization literature for convex maximization problems with polyhedra as the feasible region or a single nonlinear constraint. Additionally, a two-stage algorithm was proposed in \cite{ben2022algorithm} that computes a good initial point and then uses the gradient ascent algorithm described in \cite{luss2013conditional}. 

In this paper, we focus on the Frank-Wolfe (FW) algorithm or the conditional gradient (CG) algorithm. The FW algorithm is a famous algorithm in machine learning and optimization \cite{bertsekas1997nonlinear}. The original algorithm was proposed to minimize a quadratic function over a polytope \cite{frank1956algorithm}. Later, the algorithm was extended to more general settings \cite{levitin1966constrained,dunn1980convergence}. Yet the popularity of the algorithm in machine learning came much later \cite{hazan2012projection,jaggi2013revisiting}. We refer to \cite{yurtsever2018conditional,yurtsever2019conditional,kerdreux2020accelerating} for recent developments. 

Using the FW algorithm for convex maximization over a polyhedron with a unit step size to generate a finite algorithm that converges to a stationary point was suggested in \cite{mangasarian1996machine}. In \cite{journee2010generalized}, under the assumption of strong convexity of $g$ or strong convexity of the set $\setX$ and a lower bound on the norm of the subgradients of $g$ for the non-differentiable case, a bound on the number of iterates required to produce pair of iterates $(x_{k+1},x_{k})$ with small $\ell_2$ distance was constructed. This was generalized to Bregman distances by utilizing the relatively strong convexity concept in \cite{chaudhury2024competitiveequilibriumchoresdual}. A more generic analysis of the conditional gradient algorithm with a unit step size for convex maximization problems was given in \cite{luss2013conditional}. 

The main contribution of this paper is to prove new convergence results based on the Kurdyka-\L{}ojasiewicz (KL) inequality under strong convexity and smoothness assumptions. Using the KL property, we prove convergence of the point sequence $\{x_k\}_{k \in \NN}$ and also derive a convergence rate in some cases. To the best of our knowledge, this is the first last-iterate convergence result outside the polyhedral setting.  Next, we investigate three applications and derive a new algorithm based on this result. For the use of the KL inequality in optimization, we refer to \cite{bolte2018first,teboulle2020novel,attouch2010proximal,bolte2014proximal} and references therein. 

This paper is organized as follows. In Section \ref{sec:fw}, we briefly overview the FW algorithm and adapt it for maximizing a strongly convex function. Then we discuss the design choices within the algorithm and their consequences for the model assumptions. Then, we give the classical results regarding the FW algorithm for convex maximization problems. Moreover, we give a global convergence result under the KL assumption and also study its convergence rate. In Section \ref{sec:application}, we first give examples of the functions that satisfy the KL property. Next, we recognize previously proposed algorithms as a special case of the GFW algorithm. Finally, to our knowledge, we derive a new algorithm for the SDP relaxation of the Max-Cut problem. 

We follow standard notation and concepts which can be found in \cite{bolte2018first,luss2013conditional}.

\section{Review of Existing Results}
In this section, we briefly review the Frank-Wolfe algorithm, the existing results for the analysis of the Frank-Wolfe algorithm applied to the convex maximization setting, and findings for the gradient-like descent sequences, which will be useful for proving the global convergence result in our setting. 

\subsection{Frank-Wolfe}
\label{sec:fw}
In this section, we briefly review the FW algorithm. Consider the following optimization template

\begin{equation}
\label{eq:cg_setup}
\begin{array}{ll@{}ll}
\underset{x}\min \quad & g(x)  \\
\text{st.} \quad & x \in \setX, \\
\end{array}
\end{equation}

where $g$ is only assumed to be continuously differentiable and the set $\setX$ is a (nonempty) compact convex set. The FW algorithm applied to (\ref{eq:cg_setup}) uses the following algorithmic scheme

\begin{equation}
\label{eq:generic_cg_template}
\begin{array}{ll@{}ll}
s_k \in \underset{y\in \setX}\argmin \ \nabla g(x_k)^T(y-x_k) \\
x_{k+1} = (1-\eta_k) x_k + \eta_k s_k
\end{array}
\end{equation}

where $\eta_k$ is a step size. There are many step size rules that can be adapted, such as the Armijo Rule, the Limited Minimization Rule, and constant step size \cite{bertsekas1997nonlinear}. Standard theory shows that any limit point of the sequence $\{x_k\}_{k\in \NN}$ generated by the FW algorithm is a stationary point, and under additional assumptions, convergence rates can be derived \cite{dunn1980convergence}. Later, an additional convergence rate was proven in terms of the Frank Wolfe gap (see Equation (\ref{eq:fw_gap})) in \cite{lacoste2016convergence}.

The FW algorithm has several advantages: it scales well to large-scale optimization problems since it only requires one minimization of a linear function over the feasible region. In comparison, proximal algorithms require projection onto the set $\setX$, which can be significantly more expensive. Thus, the FW algorithm is one of the most efficient algorithms for the minimization of a function over a structured domain \cite{yurtsever2017sketchy}. In addition, the FW algorithm generates sparse iterates, which can be practical when storage is a concern \cite{jaggi2013revisiting}. Moreover, the iterates generated by the FW algorithm are naturally feasible, and thus, intermediate iterates computed can also be useful. 

\subsection{Maximizing a Strongly Convex Function with Greedy Frank-Wolfe}
\label{sec:strongly_convex_gfw}
Our focus is specifically on maximizing a strongly convex function. The upcoming analysis will show that this problem has nice properties that enables the application of simple algorithms for such problems. Furthermore, under additional assumptions such as Lipschitz gradient and KL property, stronger convergence properties can be obtained. 

The following useful property shows that the restriction to convex functions allows us to work with possibly nonconvex sets without loss because working with the convex hull yields identical results (see \cite[Section 32]{rockafellar1970convex}).

\begin{proposition}
\label{prop:cvx_hull_sup}
Let $g: \RR^d \mapsto \RR$ be a convex function and $S \subset \RR^d$ be an arbitrary set. Let $\textbf{conv}(S)$ denote its convex hull. Then
\begin{enumerate}
    \item $\sup\{g(x): x \in \textbf{conv}(S)\} = \sup\{g(x): x \in S\}$ where the first supremum is attained only if the second is attained. 
    \item if $S$ is compact, then the supremum of $g$ on $S$ is finite and it is attained at some extreme point of $S$. 
\end{enumerate}

\end{proposition}

We remark that the restriction of $g$ to be a convex function is also enough to eliminate the step size strategy and use a unit step size at each iteration. This can be seen from the gradient inequality for a convex function $g$

\begin{equation}
\label{eq:subgradient_ineq}
\begin{array}{lllll}
g(y) \geq g(x) + \nabla g(x) ^T (y-x).
\end{array}
\end{equation}

Since the FW algorithm constructs the next iterate by maximizing the gradient inner product, choosing a unit step size maximizes the lower bound given on the right-hand side. We also note that a similar statement can be shown for nondifferentiable $g$ by using subgradients. 

The above observations motivate the GFW algorithm, also known as ConGradU. This is a conditional gradient algorithm with a unit step size. 

\begin{algorithm}[ht]
\caption{Greedy Frank-Wolfe Algorithm (GFW)}
\label{alg:gfw}
\begin{algorithmic}[1]
\STATE{\textbf{Input:} $x_0 \in \setX $}
\FOR{$k = 0,1,2, ...$}
\STATE{$x_{k+1} \in \argmax \{\nabla g(x_k)^T x: x\in \setX \} $}
\ENDFOR
\end{algorithmic}
\end{algorithm}

The GFW algorithm, as presented, does not specify a stopping criterion yet; the following analysis shows there are natural candidates. For example, if the FW gap (see Equation \ref{eq:fw_gap}) or the distance between consecutive iterates is very small, the algorithm can be terminated, since further iterations may not substantially improve the objective value. 

In addition, the second statement in Proposition \ref{prop:cvx_hull_sup} shows that we can always assume the maximization step in Algorithm \ref{alg:gfw} returns an extreme point. This intuition is formalized in the following result, which helps us understand how the algorithm behaves in the later discussion.

\begin{proposition}
\label{prop:cvx_hull_iterate}
Let $g: \RR^d \mapsto \RR$ be a convex function and $\setX \subset \RR^d$ be a nonempty compact set. Given the same initial point $x_0 \in \setX$, the GFW algorithm applied to $g$ on the set $\setX$ produces identical iterates to the GFW algorithm applied to $g$ on the set $\textbf{conv}(\setX)$. 
\end{proposition}

\begin{proof}
Since the GFW algorithm constructs the next iterate by maximizing a linear function (which satisfies concavity) over a compact set, the optimal solution is attained at an extreme point by Proposition \ref{prop:cvx_hull_sup}. If the optimal solution set is not unique, then it will be the convex hull of a subset of extreme points. Since the primary problem is to maximize a convex function, the function value at an extreme point is always larger than or equal to the function value at a non-extreme point. Thus, we can consider only the extreme points of $\setX$. Therefore, the GFW algorithm applied on the set $\setX$ and its convex hull returns identical results. 
\end{proof}

\begin{figure}[ht]
\centering
\includegraphics[width=.49\textwidth]{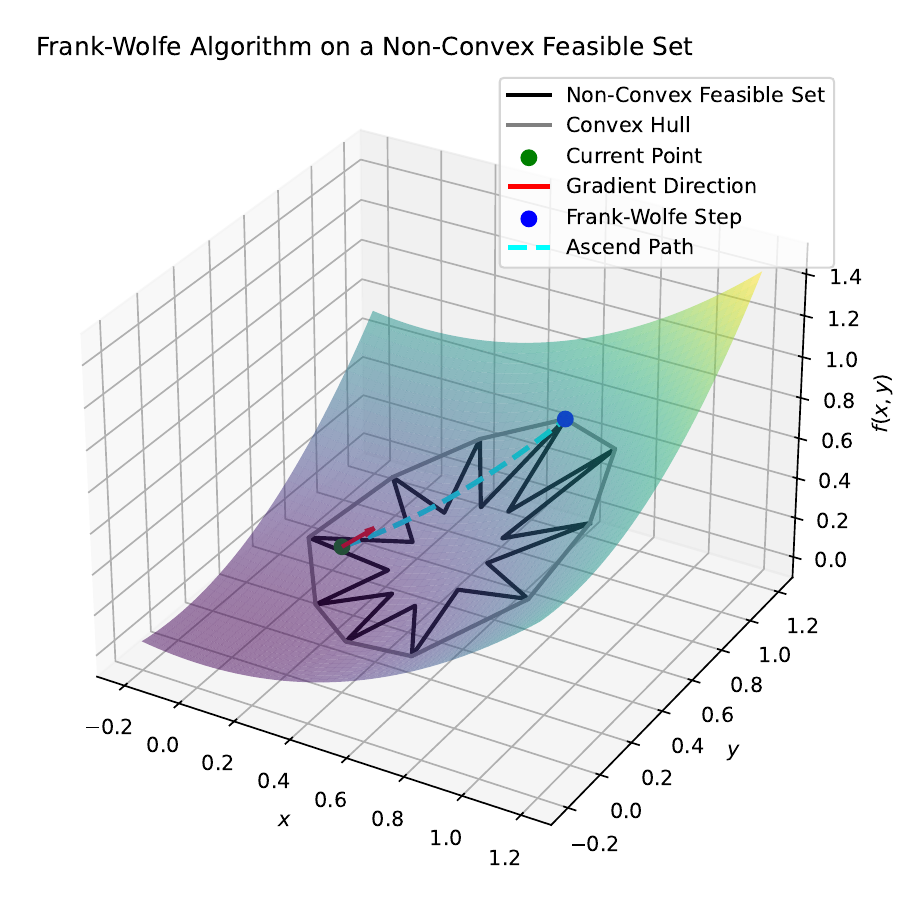}
\includegraphics[width=.49\textwidth]{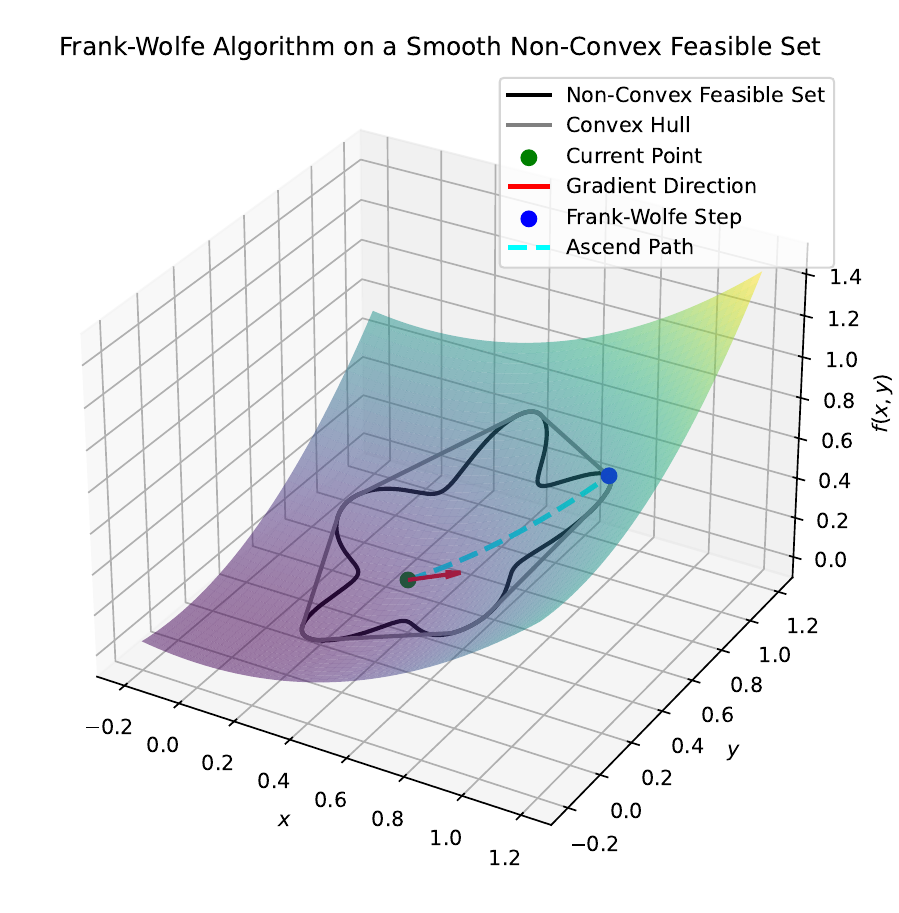}
\caption{Illustration of a single iteration of the GFW algorithm on two nonconvex sets. The objective is strictly increasing along the path of the maximizer of the gradient inner product, thus ignoring the points in between, which may or may not be in the feasible set. }
\label{fig:gfw_visual}
\end{figure}

Figure \ref{fig:gfw_visual} illustrates the behavior of the GFW algorithm for maximizing a convex function as described in Proposition \ref{prop:cvx_hull_iterate}. In each iteration, the algorithm moves to an extreme point of the feasible set. The objective is always increasing along the path, which allows us to use a unit step size and work on non-convex sets. The same property does not hold for concave maximization problems; hence, a suitable step size strategy and convexity of the feasible set are essential.

We next present a duality gap measure $\gamma(x)$, known as the FW gap, which is particularly suited for analyzing FW-type algorithms \cite{jaggi2013revisiting,lacoste2016convergence,luss2013conditional,journee2010generalized,leplat2023conic}

\begin{equation}
\label{eq:fw_gap}
\begin{array}{lllll}
\gamma(x) := \max\{ \nabla g(x)^T(y-x) : y \in \setX \}
\end{array}
\end{equation}

Since $\setX$ is compact, $\gamma(x)$ is well-defined and admits a global maximizer. 

The following lemma is useful for proving the main convergence result.

\begin{lemma}
\label{lemma:iterate_connection}
Let $g: \RR^d \mapsto \RR$ be a strongly convex function, and $\setX \subset \RR^d$ be a nonempty compact set, and $\gamma(x)$ is defined as in (\ref{eq:fw_gap}). Then the sequence $\{x_{k}\}_{k\in \NN}$ generated by the GFW algorithm satisfies
\begin{align}
\gamma(x_k) &\geq 0 &\forall k \in \NN  \label{eq:fw_gap_nonneg} \\
g(x_{k+1}) - g(x_k) &\geq \gamma(x_k) + \alpha \|x_{k+1}-x_k\|_2^2 \quad &\forall k \in \NN \label{eq:iterate_connection} 
\end{align}
\end{lemma}

\begin{proof}

The proof of (\ref{eq:fw_gap_nonneg}) follows from the definition of $\gamma(x)$. Similarly, the proof of (\ref{eq:iterate_connection}) follows from the gradient inequality for the strongly convex function $g$.

\begin{equation}
\label{eq:iterate_connection1}
\begin{array}{lllll}
g(y) \geq  g(x) + \nabla g(x)^T(y-x) + \alpha \|y-x\|_2^2
\end{array}
\end{equation}

Plugging in $y = x_{k+1}$ and using the definition of $\gamma(x_k)$ and construction of $x_{k+1}$, we get the desired result. 
\end{proof}

Using this lemma, we are ready to present initial convergence results. These initial convergence results are similar in nature to those obtained in other papers studying FW-like algorithms for convex maximization. We will show stronger convergence results in Theorem~\ref{thm:global_convergence_gfw} with the KL assumption.

\begin{theorem}
\label{thm:classic_convergence}
Let $g: \RR^d \mapsto \RR$ be a strongly convex and smooth function, and $\setX \subset \RR^d$ be a nonempty compact set. Let $\{x_{k}\}_{k\in \NN}$ be the sequence generated by the GFW algorithm. Then the following statements hold.
\begin{enumerate}
    \item The sequence of function values $\{g(x_{k})\}_{k\in \NN}$ is monotonically increasing and 
    \begin{align}
    \lim_{k \to \infty} \gamma(x_k) + \alpha\|x_{k+1}-x_k\|_2^2 = 0
    \end{align}

    \item Either for some $k$, the iterate satisfies $\gamma(x_k) = 0$ and the algorithm repeats $x_k$ indefinitely. Otherwise, the algorithm generates an infinite sequence $\{x_{k}\}_{k\in \NN}$ with strictly increasing function values $\{g(x_{k})\}_{k\in \NN}$.

    \item Every limit point of the sequence $\{x_{k}\}_{k\in \NN}$ converges to a stationary point. In other words, any limit point $\Bar{x}$ of the sequence $\{x_{k}\}_{k\in \NN}$ satisfies the following:

    \begin{equation}
    \label{eq:limit_point_stationary}
    \begin{array}{llll}
    \nabla g(\Bar{x})^T(z - \Bar{x}) \leq 0 \quad \forall z \in \setX
    \end{array}
    \end{equation}
    
\end{enumerate}
\end{theorem}

\begin{proof}
Lemma \ref{lemma:iterate_connection} shows that the sequence $\{g(x_{k})\}_{k\in \NN}$ is non-decreasing. Summing up the inequalities given by (\ref{eq:iterate_connection}) we get
\begin{equation}
\label{eq:iterate_connection2}
\begin{array}{lllll}
g(x_{k+1}) - g(x_0) \geq \sum\limits_{i=0}^k (\gamma(x_k)  + \alpha \|x_{k+1}-x_k\|_2^2).
\end{array}
\end{equation}
Since $\setX$ is compact and $g$ is continuous, $g(x_k)$ is bounded above by the global maximum, say $g^* = \max\{g(x): x\in \setX \}$. 
Therefore the LHS in \eqref{eq:iterate_connection2} is bounded above by a constant independent of $k$, and the nonnegative series $\sum\limits_{i=0}^{\infty} (\gamma(x_k)  + \alpha \|x_{k+1}-x_k\|_2^2)$ is convergent. It follows that $\gamma(x_k)  + \alpha \|x_{k+1}-x_k\|_2^2$ converges to zero. 

If $\gamma(x_k) = 0$ for some $k$, then the algorithm repeats $x_k$ for the rest of the iterations since it satisfies the stationarity condition given by (\ref{eq:limit_point_stationary}). Otherwise, the algorithm generates strictly increasing function values $g(x_{k+1}) > g(x_k)$. 

Finally, assume that a limit point $\Bar{x}$ does not satisfy the stationarity condition. Then we must have for some $z \in \setX$

\begin{equation}
\label{eq:limit_point_stationary1}
\begin{array}{lllll}
\delta = \nabla g(\Bar{x})^T(z-\Bar{x}) > 0.
\end{array}
\end{equation}

Consider a subsequence $\{x_{n_k}\}_{k\in \NN}$ that converges to $\Bar{x}$. Since the function has a Lipschitz continuous gradient, for sufficiently large $k$ we have

\begin{equation}
\label{eq:limit_point_stationary2}
\begin{array}{lllll}
\nabla g(x_{n_k})^T(z-x_{n_k}) > \frac{\delta}{2}.
\end{array}
\end{equation}

However, this is an immediate contradiction to $\gamma(x_k)$ converging to zero. Thus, we conclude that any limit point is stationary. 
\end{proof}

The statement and proof of Theorem \ref{thm:classic_convergence} is styled after \cite{luss2013conditional}, where convergence analysis is given for general convex functions without the assumption of strong convexity or smoothness. We include an adapted proof for our setting with additional structural assumptions on the problem data, which allow us to show that the FW gap plus the squared distance between the consecutive iterates goes to zero, rather than just the FW gap.

Remark that in the convergence proof of the first statement of Theorem \ref{thm:classic_convergence}, specifically the result $\gamma(x_k) + \alpha \| x_k-x_{k+1} \|_2^2\to 0$, we show the summability of $\sum\limits_{i=0}^{\infty} (\gamma(x_k)  + \alpha \|x_{k+1}-x_k\|_2^2)$. From this fact, we can derive a convergence rate on the minimum occurrence of the FW gap plus the squared distance between the consecutive iterates. In particular, it decreases with a rate of at least $O(1/k)$. This result and a variant under different assumptions were shown in \cite{journee2010low}. This was shown on the FW gap for convex objectives (without strong convexity) in \cite{yurtsever2022cccp}.  A more general result without assuming convexity or concavity but under a bounded curvature assumption was shown in \cite{lacoste2016convergence}, where they prove that the minimum FW gap decreases at a rate of $O(1/\sqrt{k})$.

\subsection{Gradient-Like Descent Sequences}
\label{sec:grad_descent}
In this section, we briefly review the concept of a gradient-like descent sequence introduced in \cite{bolte2014proximal} and refined convergence results shown in \cite{bolte2018first}. In order to be consistent with the existing literature, we consider the following minimization problem in this section:

\begin{equation}
\label{eq:inf_F}
\begin{array}{ll@{}ll}
\underset{x \in \RR^d}\inf F(x),
\end{array}
\end{equation}

where $F: \RR^d \mapsto (-\infty,\infty]$ is a proper lower semicontinuous function that is bounded from below. 

We use the notion of limiting subdifferential and Fermat's rule \cite[Definition 8.3]{rockafellar2009variational}, which characterizes the set of critical points of $F$ as

\begin{equation}
\label{eq:set_of_critical}
\begin{array}{ll@{}ll}
\text{crit}F = \{x \in \RR^d : 0 \in \partial F(x) \}.
\end{array}
\end{equation}

Here, $\partial F(x)$ denotes the limiting subdifferential of $F(x)$ at $x$. To prove a global convergence result for a sequence $\{x_{k}\}_{k\in \NN}$ to a critical point of $F$, the following abstract notion is introduced.

\begin{definition}(gradient-like descent sequence \cite{bolte2018first})
\label{def:grad_descent_seq}
Let $F: \RR^d \mapsto (-\infty,\infty]$ be a proper lower semicontinuous function. A sequence $\{x_{k}\}_{k\in \NN}$ is called a gradient-like descent sequence for minimizing $F$ if the following three conditions are met:
\begin{enumerate}[label=(C\arabic*),leftmargin=*]
    \item Sufficient decrease property. There exists a positive constant $\rho_1$ such that
    
    \begin{equation}
    \label{eq:suff_decrease}
    \begin{array}{ll@{}ll}
    F(x_k) - F(x_{k+1}) \geq  \rho_1 \|x_{k+1}-x_k\|_2^2 \quad \forall k \in \NN.
    \end{array}
    \end{equation}
    
    \item Subgradient lower bound for the iterates gap. There exists $w_{k+1} \in \partial F(x_{k+1})$ and a positive constant $\rho_2$ such that 
    
    \begin{equation}
    \label{eq:subgradient_lower_bound}
    \begin{array}{ll@{}ll}
    \|w_{k+1}\|_2 \leq \rho_2 \|x_{k+1} - x_k\|_2 \quad \forall k \in \NN.
    \end{array}
    \end{equation}

    \item Let $\Bar{x}$ be a limit point of the sequence $\{x_k\}_{k\in \NN}$. Then, $\underset{k \to \infty}\limsup \ F(x_k) \leq F(\Bar{x})$.
\end{enumerate}
\end{definition}

The conditions (C1) and (C2) are usually verified easily for any descent-type algorithm \cite{attouch2009convergence}. The condition (C3) is also a mild requirement, as it is automatically implied when $F$ is continuous.

The concept of gradient-like descent sequence is a powerful abstraction tool that allows us to prove convergence of algorithms such as the proximal alternating linearized minimization (PALM) algorithm introduced in \cite{bolte2014proximal} for nonconvex and nonsmooth problems under the KL assumption. Additionally, the Bregman proximal gradient (BPG) introduced in \cite{bauschke2017descent} was shown to be globally convergent for nonconvex and nonsmooth problems under the KL condition in \cite{bolte2018first}. Now we provide a formal definition of the KL property.

\begin{definition}(Kurdyka-\L{}ojasiewicz property)
\label{def:kl_property}
The function $F$ is said to have the Kurdyka-\L{}ojasiewicz property locally at $\Bar{x} \in \dom \partial F$ if there exists $\eta \in (0,\infty]$, a neighborhood of $U$ of $\Bar{x}$ and a continuous concave function $\varphi: [0,\eta) \mapsto \RR_+$ such that:
\begin{itemize}

    \item $\varphi(0) = 0$,
    
    \item $\varphi$ is $C^1$ on $(0,\eta)$
    
    \item for all $s \in (0,\eta)$, $\varphi'(s) > 0$,
    
    \item and for all $x$ in $U \cap [F(\Bar{x}) < F < F(\Bar{x}) + \eta]$ the Kurdyka-\L{}ojasiewicz inequality holds
    
    \begin{align}
        \varphi'(F(x)-F(\Bar{x}))\dist(0,\partial F(x)) \geq 1
    \end{align}
    
\end{itemize}
\end{definition}

If $F$ has the KL property at each point in the domain of $\partial F$, then $F$ is called a KL function or $F$ satisfies the KL property. It is non-trivial to check the KL property for a given function. However, in practice, many functions satisfy the KL property. In Section \ref{sec:application} we recall some important classes of functions that satisfy the KL property, give examples of such functions, and show applications in selected optimization problems. 

Now we recall two important results for gradient-like descent sequences under the KL property: global convergence of the sequence to a critical point and a rate of convergence. 

\begin{theorem}(Global convergence \cite{bolte2018first})
\label{thm:global_convergence}
Let $\{x_{k}\}_{k\in \NN}$ be a bounded gradient-like descent sequence for minimizing $F$. If $F$ satisfies the KL property, then the sequence $\{x_{k}\}_{k\in \NN}$ has finite length, i.e., $\sum\limits_{k=1}^\infty \|x_{k+1}-x_k\| < \infty$, and it converges to $x^* \in \text{crit}F$.
\end{theorem}

\begin{theorem}(Convergence rate \cite{bolte2018first})
\label{thm:convergence_rate}
Let $\{x_{k}\}_{k\in \NN}$ be a bounded gradient-like descent sequence for minimizing $F$. Assume that $F$ satisfies the KL property, where the desingularizing function $\varphi$ of $F$ is of the following form
\begin{equation}
\label{eq:desingular_function}
\begin{array}{ll@{}ll}
\varphi(s) = cs^{1-\theta}, \quad c > 0, \theta \in [0,1).
\end{array}
\end{equation}
Let $\Bar{x}$ be the limit point of $\{x_{k}\}_{k\in \NN}$. Then the following convergence rates hold:
\begin{enumerate}[leftmargin=*]
    \item If $\theta = 0$, then the sequence $\{x_{k}\}_{k\in \NN}$ converges in a finite number of steps.

    \item If $\theta \in (0,0.5]$, then there exist constants $\omega > 0$ and $\tau \in [0,1)$, such that
    
    \begin{equation}
    \label{eq:linear_convergence}
    \begin{array}{ll@{}ll}
    \|x_k-\Bar{x}\| \leq \omega \tau^k.
    \end{array}
    \end{equation}

    \item If $\theta \in (0.5,1)$, then there exist constants $\omega > 0$ such that

    \begin{equation}
    \label{eq:sublinear_convergence}
    \begin{array}{ll@{}ll}
    \|x_k-\Bar{x}\| \leq \omega k^{-\frac{1-\theta}{2\theta -1}}.
    \end{array}
    \end{equation}
\end{enumerate}
\end{theorem}

\section{Global Convergence of GFW under the KL property}
In this section, we extend the convergence results presented in Section \ref{sec:strongly_convex_gfw} under the assumption of the KL property.  

For the global convergence proof, we reformulate our problem as a nonconvex minimization problem as in Section \ref{sec:grad_descent} in order to leverage convergence results obtained for gradient-like descent sequences. Rewriting model (\ref{eq:optimization_model}) gives

\begin{equation}
\label{eq:optimization_model_canon}
\begin{array}{ll@{}ll}
\underset{x}\min \{F(x) =  f(x) + \I_{\setX}(x) \} \\
\end{array}
\end{equation}
\noindent
where $f(x) = - g(x)$ and $\I_{\setX}(x)$ is the indicator function of the set $\setX$. For this problem, the set of critical points can be characterized as

\begin{equation}
\label{eq:set_of_critical_concave_min}
\begin{array}{ll@{}ll}
\text{crit}F = \{x \in \RR^d : 0 \in \partial F(x) \equiv \nabla f(x) + \partial \I_{\setX}(x)  \}.
\end{array}
\end{equation}

Our global convergence result relies on two pillars. Firstly, we show that the sequence generated by the GFW algorithm for problem (\ref{eq:optimization_model}) under Assumption \ref{as:convex_smooth_compact} is a gradient-like descent sequence (see Definition \ref{def:grad_descent_seq}) for minimizing $F$ in model (\ref{eq:optimization_model_canon}). The second is an assumption on the problem data $f$ and $\setX$, namely, that it satisfies the KL property. Then we use the convergence results established for bounded gradient-like descent sequences as recalled from the literature in Section \ref{sec:grad_descent}. 

We will now establish that the sequence generated by the GFW algorithm is a gradient-like descent sequence for $F$.

\begin{lemma}
\label{lemma:gfw_grad_descent}
Consider problem (\ref{eq:optimization_model}) under Assumption \ref{as:convex_smooth_compact}.  Then the sequence $\{x_{k}\}_{k\in \NN}$ generated by the GFW algorithm for solving problem (\ref{eq:optimization_model}) is a gradient-like descent sequence for minimizing $F$ given in (\ref{eq:optimization_model_canon}).
\end{lemma}

\begin{proof}

Using Lemma \ref{lemma:iterate_connection} and $g = -f$ we can write

\begin{equation}
\label{eq:suff_decrease1}
\begin{array}{ll@{}ll}
f(x_{k}) - f(x_{k+1}) \geq \alpha\|x_{k+1}-x_k\|_2^2.
\end{array}
\end{equation}
Using feasibility of the sequence $\{x_k\}_{k\in \NN}$, we arrive at condition (C1) by taking $\rho_1 = \alpha$.

Now, we reformulate construction of $x_{k+1}$ as
\begin{equation}
\label{eq:subgradient_lower_bound0}
\begin{array}{ll@{}ll}
x_{k+1} \in \underset{y}\argmin \{ \nabla f(x_k)^Ty + \I_{\setX}(y)\}
\end{array}
\end{equation}
where we used $\argmax \{g\} = \argmin\{-g\}$ and lifted the set constraint to the objective by using an indicator function. Then, writing the optimality condition for $x_{k+1}$ we get

\begin{equation}
\label{eq:subgradient_lower_bound1}
\begin{array}{ll@{}ll}
0 \in \nabla f(x_k) + \partial \I_{\setX}(x_{k+1})
\end{array}
\end{equation}

\noindent
Therefore, by defining

\begin{equation}
\label{eq:subgradient_lower_bound2}
\begin{array}{ll@{}ll}
w_{k+1} = \nabla f(x_{k+1}) - \nabla f(x_k),
\end{array}
\end{equation}

\noindent
we see that $w_{k+1} \in \partial F(x_{k+1})$. Using the smoothness of $f$ we have

\begin{equation}
\label{eq:subgradient_lower_bound3}
\begin{array}{ll@{}ll}
\|w_{k+1}\|_2 \leq L\|x_{k+1}-x_k\|_2.
\end{array}
\end{equation}

\noindent
Hence, we verify that the condition (C2) holds for $\rho_2 = L$. 

Consider a subsequence $\{x_{n_k}\}_{k\in \NN}$ that converges to some $\Bar{x} \in \setX$ (This occurs by compactness). Since the iterates and the limit point are feasible, $F(x_k) = f(x_k) \ \forall k \in \NN$ and $F(\Bar{x}) = f(\Bar{x})$. Furthermore, by Theorem \ref{thm:classic_convergence} we have that $f(x_k)$ decreases down to a limit, and $f$ is continuous. Combining all these facts yields

\begin{equation}
\label{eq:function_limit1}
\begin{array}{ll@{}ll}
\underset{k \to \infty}\lim F(x_{n_k}) = \underset{k \to \infty}\lim f(x_{n_k}) = f(\Bar{x}) = F(\Bar{x})
\end{array}
\end{equation}

\noindent
Consequently, we obtain condition (C3). 

\end{proof}

Now we are ready to establish a global convergence result under the KL property. 

\begin{theorem}(Global convergence of GFW)
\label{thm:global_convergence_gfw}
Consider problem (\ref{eq:optimization_model_canon}) under Assumption \ref{as:convex_smooth_compact} with the additional assumption that $F$ satisfies the KL property. Let $\{x_{k}\}_{k\in \NN}$ be a sequence generated by the GFW algorithm for maximizing (\ref{eq:optimization_model}). Then the sequence $\{x_{k}\}_{k\in \NN}$ converges to some $x^*$ that lies in the set $\text{crit}F$.
\end{theorem}

The proof of this theorem follows from Lemma \ref{lemma:gfw_grad_descent} and Theorem \ref{thm:global_convergence}. Moreover, the statement can be slightly sharpened by using Proposition \ref{prop:cvx_hull_iterate}. Since the sequence generated by applying the algorithm instead to the convex hull of $\setX$ does not change anything, we can consider the following refinement of the set of critical points

\begin{equation}
\label{eq:set_of_critical_extreme}
\begin{array}{ll@{}ll}
\text{crit}^*F = \{x \in \RR^d : 0 \in \partial F(x) \equiv \nabla f(x) + \partial \I_{\textbf{conv}(\setX)}(x)  \}.
\end{array}
\end{equation}

In this definition, $\partial \I_{\textbf{conv}(\setX)}(x)$ reduces to the classical subdifferential since it is an indicator function of a convex set. Then the statement of Theorem \ref{thm:global_convergence_gfw} holds with $x^* \in \text{crit}^*F$. This observation shows that the GFW algorithm will not converge to a stationary point that is not an extreme point of the set $\setX$. 

Additionally, we also present a result on the rate of convergence which applies under the same conditions as Theorem \ref{thm:global_convergence_gfw}.

\begin{theorem}(Convergence rate of GFW)
\label{thm:convergence_rate_gfw}
Consider problem (\ref{eq:optimization_model_canon}) under Assumption \ref{as:convex_smooth_compact}. Let $\{x_{k}\}_{k\in \NN}$ be a sequence generated by the GFW algorithm for maximizing (\ref{eq:optimization_model}). Assume further that $F$ satisfies the KL property, where the desingularizing function $\varphi$ of $F$ has the following form
\begin{equation}
\label{eq:desingular_function_gfw}
\begin{array}{ll@{}ll}
\varphi(s) = cs^{1-\theta}, \quad c > 0, \theta \in [0,1).
\end{array}
\end{equation}
Let $\Bar{x}$ be the limit of the sequence $\{x_{k}\}_{k\in \NN}$. Then the following convergence rates hold:
\begin{enumerate}[leftmargin=*]
    \item If $\theta = 0$, then the sequence $\{x_{k}\}_{k\in \NN}$ converges in a finite number of steps.

    \item If $\theta \in (0,0.5]$, then there exist constants $\omega > 0$ and $\tau \in [0,1)$, such that
    
    \begin{equation}
    \label{eq:linear_convergence_gfw}
    \begin{array}{ll@{}ll}
    \|x_k-\Bar{x}\| \leq \omega \tau^k.
    \end{array}
    \end{equation}

    \item If $\theta \in (0.5,1)$, then there exist constants $\omega > 0$ such that

    \begin{equation}
    \label{eq:sublinear_convergence_gfw}
    \begin{array}{ll@{}ll}
    \|x_k-\Bar{x}\| \leq \omega k^{-\frac{1-\theta}{2\theta -1}}.
    \end{array}
    \end{equation}
\end{enumerate}
\end{theorem}

Similarly, the proof of this theorem follows from Lemma \ref{lemma:gfw_grad_descent} and Theorem \ref{thm:convergence_rate}. 

We remark that the convergence rate reported by Theorems \ref{thm:convergence_rate} and \ref{thm:convergence_rate_gfw} are asymptotic in the sense that, it is difficult to estimate the exponent $\theta$ of the desingularizing function $\varphi$ and the parameters $\omega,\tau$ for a given problem. Furthermore, for the particular instance when the feasible set $\setX$ is polyhedral, Theorem \ref{thm:convergence_rate_gfw} does not give a new convergence result; finite convergence to a stationary point in this setting was first pointed out in \cite{mangasarian1996machine}. In Appendix \ref{ap:local_min_polyhedra}, we improve this statement as follows: under the assumption that we can check alternative optimal solutions when working over a polytope, the GFW algorithm converges to a strict local minimum in a finite number of steps since in our setting, $g$ is assumed to be strongly convex.

\section{Applications}
\label{sec:application}

In this section, we show various applications of the GFW algorithm on strongly convex maximization problems thanks to the rich family of KL functions. Given a function, it might be hard to check if it satisfies the KL property. However, for a large class of functions, it holds true. In \cite{bolte2007lojasiewicz}, it was proven that it holds for the class of semialgebraic functions. A non-exhaustive list of semialgebraic functions and sets is as follows:

\begin{itemize}
    \item Real polynomial functions.
    \item Indicator functions of semialgebraic sets.
    \item Finite sums and products of semialgebraic functions.
    \item $\ell_p$ norms when $p>0$ is rational and the $\ell_0$ norm.
    \item Cone of positive semidefinite matrices, Stiefel manifold. 
\end{itemize}

Furthermore, it was shown in \cite{bolte2007clarke} that the class of definable functions satisfies the KL property. This deep result shows that the KL property holds for subanalytic functions, in particular globally subanalytic functions, and functions definable on the log-exp structure. 

Moreover, functions that are semialgebraic or globally subanalytic satisfy the KL property with a desingularizing function of the form $\varphi(s) = cs^{1-\theta}$ for some $\theta \in [0,1) \cap \mathbb{Q}$. 

For the proofs of these statements and additional properties, we refer to \cite{attouch2009convergence,attouch2010proximal,attouch2013convergence,bolte2014proximal}.

\subsection{Reweighted Optimization for Minimizing the \texorpdfstring{$\ell_0$}{l0} Norm}

An important problem in signal processing is to recover a sparse solution from a small number of observations. Mathematically speaking, the problem can be formulated as model (\ref{eq:compressed_sensing}) (see \cite{donoho2006compressed,candes2008enhancing,bruckstein2009sparse}). In this model, the $\ell_0$ norm $\|x\|_0$ counts the number of nonzero components of $x \in \RR^n$, $A$ is a given matrix $A \in \RR^{m\times n}$, and observation vector $b\in \RR^m$. This problem is computationally intractable because of its combinatorial nature due to the sparsity constraint. Therefore, proxy alternatives to problem (\ref{eq:compressed_sensing}) are often used. We focus on one such approach, which aligns with our work. In \cite{candes2008enhancing}, the $\ell_0$ norm is replaced by a regularized logarithm, which results in the model (\ref{eq:rwl1}). 

\begin{minipage}[ht]{0.45\textwidth}
\begin{equation}
\label{eq:compressed_sensing}
\begin{array}{cccc}
\underset{x}\min & \|x\|_0  \\
\text{st.} & Ax = b,
\end{array}
\end{equation}
\end{minipage}
\hfill 
\begin{minipage}[ht]{0.45\textwidth}
\begin{equation}
\label{eq:rwl1}
\begin{array}{rl}
\underset{x}\min & \sum\limits_{i=1}^n\log(\epsilon + x_i^+ + x_i^-)  \\
\text{st.} & Ax^+ - Ax^- = b \\
& x^+ \geq 0, x^- \geq 0.
\end{array}
\end{equation}
\end{minipage}

In model (\ref{eq:rwl1}), the $\ell_0$ norm $\|x\|_0$ is approximated by the regularized logarithm $\log(\epsilon + |x|)$. This approach is called reweighted $\ell_1$ minimization, abbreviated as RWL1.  However, since the absolute value is not differentiable, a smooth coupling trick is applied by decomposing $x$ into positive and negative parts, i.e. $x_i^+ = \max(0,x_i)$ and $x_i^- = \max(0,-x_i)$. A similar trick was also used in the LP reformulation of the $\ell_1$ norm relaxation of the $\ell_0$ norm in \cite{donoho2006compressed}. However, the model (\ref{eq:rwl1}) is still not convex. In fact, the new objective is a concave function. Therefore, instead of solving the problem directly, an iterative algorithm is proposed to minimize the function locally in \cite{candes2008enhancing}. Let $\setX = \{(x^+,x^-) \in \RR^{2n}: Ax^+-Ax^- = b, x^+ \geq 0, x^- \geq 0 \}$ denote the feasible region. Starting from an initial guess $x_0 = (x_0^+,x_0^-) \in \RR^{2n}$, the algorithm generates a sequence of vectors by the following recursion;

\begin{equation}
\label{eq:rwl1_algorithm}
\begin{array}{ll@{}ll}
x_{n+1} \in \underset{x}\argmin \left\{ \sum_{i=1}^{n}\dfrac{x^+_i + x^-_i}{(x^+_n)_i + (x^-_n)_i + \epsilon} :  (x^+,x^-) \in \setX \right\}.\\
\end{array}
\end{equation}

One can recognize the iterative algorithm described in (\ref{eq:rwl1_algorithm}) as the GFW algorithm applied to model (\ref{eq:rwl1}) with $f = \sum\limits_{i=1}^n\log(\epsilon + x_i^+ + x_i^-), \setX = \{(x^+,x^-) \in \RR^{2n}: Ax^+-Ax^- = b, x^+ \geq 0, x^- \geq 0\}$. The convergence results reported in \cite{candes2008enhancing} were limited to a monotonic decrease of $f(x_k)$ to a limit. Our convergence results do not apply immediately to the RWL1 algorithm since this formulation violates Assumption \ref{as:convex_smooth_compact} because the function $f$ is not strongly concave. To guarantee last-iterate convergence, an alternating minimization algorithm with proximal regularization was proposed in \cite{attouch2010proximal}, which we will refer to as RWL1 Prox in the upcoming discussion. To remedy the convergence issue, instead of proximal regularization, we propose the following minor modification to model (\ref{eq:rwl1})

\begin{equation}
\label{eq:rwl1_star}
\begin{array}{rl}
\underset{x}\min & \sum\limits_{i=1}^n\log(\epsilon + x_i^+) + \log(\epsilon + x_i^-) \\
\text{st.} & Ax^+ - Ax^- = b \\
& x^+ \geq 0, x^- \geq 0.
\end{array}
\end{equation}

In this formulation, the variables are no longer coupled inside the logarithm function, and we will argue that the function is strongly convex. Now applying the GFW algorithm to model (\ref{eq:rwl1_star}) gives;

\begin{equation}
\label{eq:rwl1_star_algorithm}
\begin{array}{ll@{}ll}
x_{n+1} \in \underset{x}\argmin \left\{ \sum_{i=1}^{n}\dfrac{x^+_i}{(x^+_n)_i + \epsilon}+\dfrac{x^-_i}{ (x^-_n)_i + \epsilon} :  (x^+,x^-) \in \setX \right\}.\\
\end{array}
\end{equation}

Decoupling variables in this fashion changes the penalization term. In the algorithm described by  (\ref{eq:rwl1_algorithm}), in each iteration, the positive and negative part of $x$ is penalized with the same weight, while the second algorithm (\ref{eq:rwl1_star_algorithm}) penalizes them separately with different weights. While this modification is very minor and can be justified for theoretical convergence purposes, it can also be derived from decoupling variables in the original formulation. Decomposing $x$ into positive and negative parts in model (\ref{eq:compressed_sensing}), we arrive at the following model 

\begin{equation}
\label{eq:compressed_sensing_decomposed}
\begin{array}{cccc}
\underset{x}\min & \|x^+\|_0 + \|x^-\|_0  \\
\text{st.} & Ax^+-Ax^- = b \\
& x^+ \geq 0, x^- \geq 0.
\end{array}
\end{equation}

The optimal solution set of problems (\ref{eq:compressed_sensing}) and (\ref{eq:compressed_sensing_decomposed}) is equivalent with the correspondence $x = x^+-x^-$. From here, replacing the $\ell_0$ norm with $\log(\epsilon + x)$ (considering the non-negativity) gives the model (\ref{eq:rwl1_star}). 

Let us define $f = \log(\epsilon + x_i^+) + \log(\epsilon + x_i^-)$ and $\setX$ as before. Since $x^+$ and $x^-$ are nonnegative, $f$ is a Lipschitz continuous function since the Hessian is bounded. Furthermore, since the GFW algorithm is a descent (ascent for maximization) algorithm by Theorem \ref{thm:classic_convergence}, and $f$ has bounded level sets, $(x_k^+,x_k^-)$ will remain in a bounded set. Thus, we can add an additional non-binding constraint $\|(x_k^+,x_k^-)\|_2 \leq M$ for sufficiently large $M > 0$ without any loss. This allows us to claim that $\setX$ is compact since it is bounded and closed. Additionally, since $(x_k^+,x_k^-)$ will remain in a bounded set, the Hessian of $f$ is bounded below by some negative multiple of identity, which gives strong concavity, i.e., $-f$ is strongly convex. Thus, Assumption \ref{as:convex_smooth_compact} is satisfied. Moreover, $F = f + \I_{\setX}(x)$ satisfies the KL property since it is definable in the log-exp structure. Furthermore, it is globally subanalytic on bounded boxes, so its desingularizing functions are in the form of $\varphi = cs^\theta$ for some $c > 0$ and $\theta \in (0,1]$. Therefore, the behavior of the sequence generated by Algorithm (\ref{eq:rwl1_star_algorithm}) is governed by Theorem \ref{thm:global_convergence_gfw} and Theorem \ref{thm:convergence_rate_gfw}.

We numerically compare the original reweighted $\ell_1$ minimization algorithm \cite{candes2008enhancing} (referred as RWL1 in figures), the proximal regularized version of the reweighted $\ell_1$ minimization algorithm proposed in \cite{attouch2010proximal} (referred as RWL1 Prox in figures) and the algorithm described in (\ref{eq:rwl1_star_algorithm}) which we call split reweighted $\ell_1$ minimization algorithm (referred as RWL1 Split in figures). We use the same benchmark used in \cite{candes2008enhancing}. First, we generate a sparse signal $x \in \RR^{256}$ with cardinality $\|x\|_0 = s$. Nonzero indices of $x$ are chosen randomly, and each nonzero value is sampled from a standard normal distribution. Then we generate a matrix $A \in \RR^{100 \times 256}$ where $A_{ij} \sim \mathcal{N}(0,1)$ and then normalize each column to have a unit norm. Then, we set $b = Ax$. We fix the $\epsilon = 0.1$ and run the algorithms until the $\ell_2$ distance between consecutive iterates is less than or equal to $10^{-3}$. We initialize each algorithm with the unweighted $\ell_1$ norm solution. Finally, we repeat this experiment $200$ times for changing sparsity levels $s = 20,22,...60$. 

We use MOSEK \cite{mosek} through the CVXPY interface \cite{diamond2016cvxpy} to solve linear programs for the conditional gradient steps of the RWL1 and the RWL1 Split algorithms, and quadratic programs for the proximal steps of the RWL1 Prox algorithm. 

\begin{figure}[ht]
\centering
\includegraphics[width=.325\textwidth]{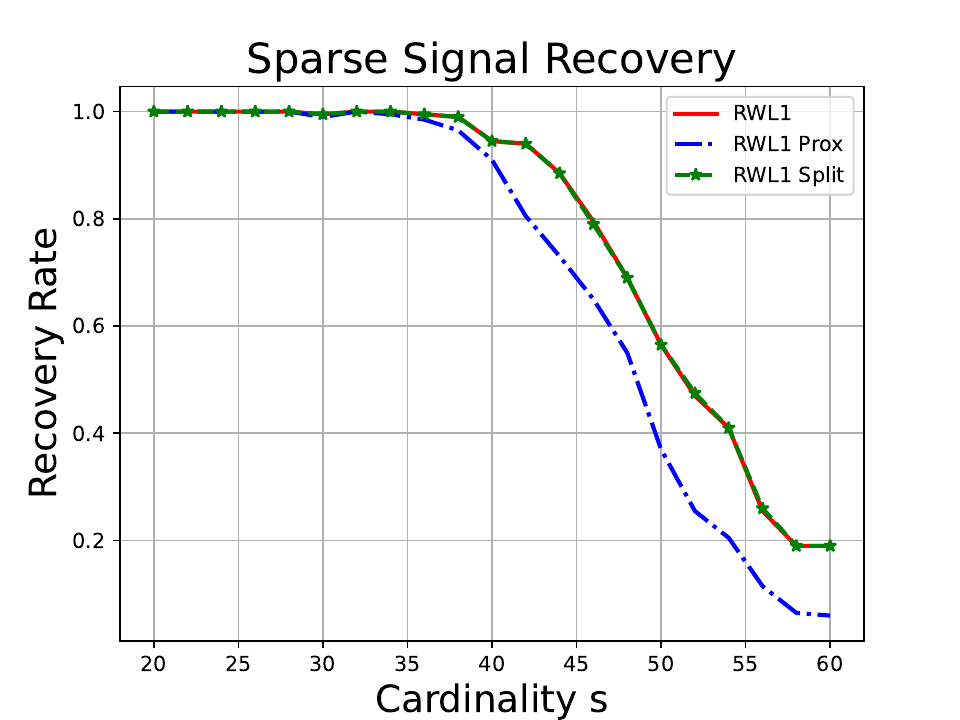}
\includegraphics[width=.325\textwidth]{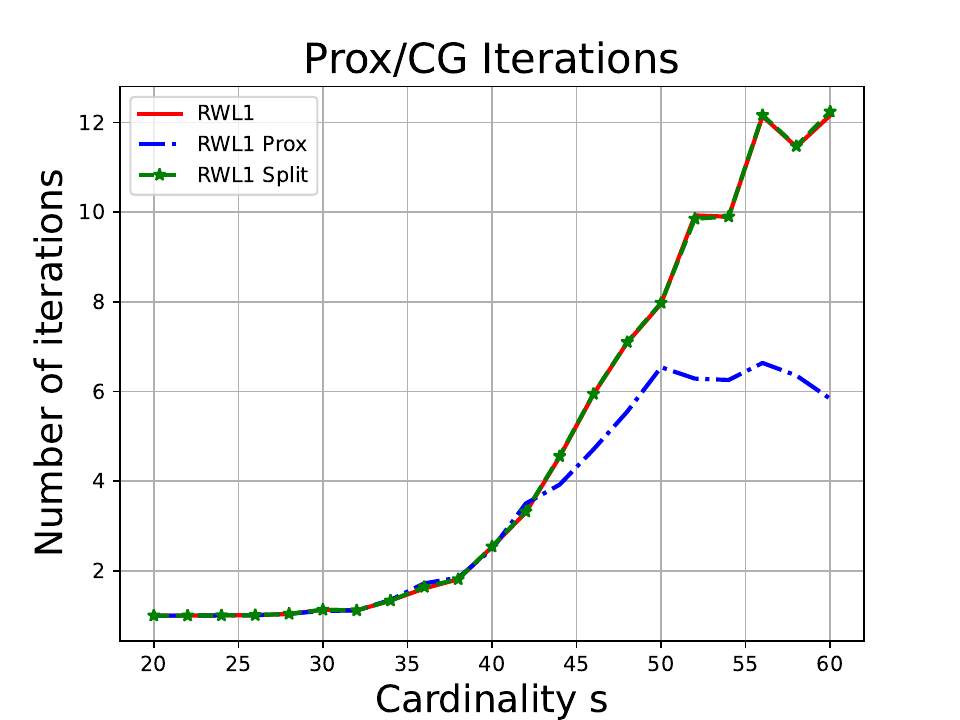}
\includegraphics[width=.325\textwidth]{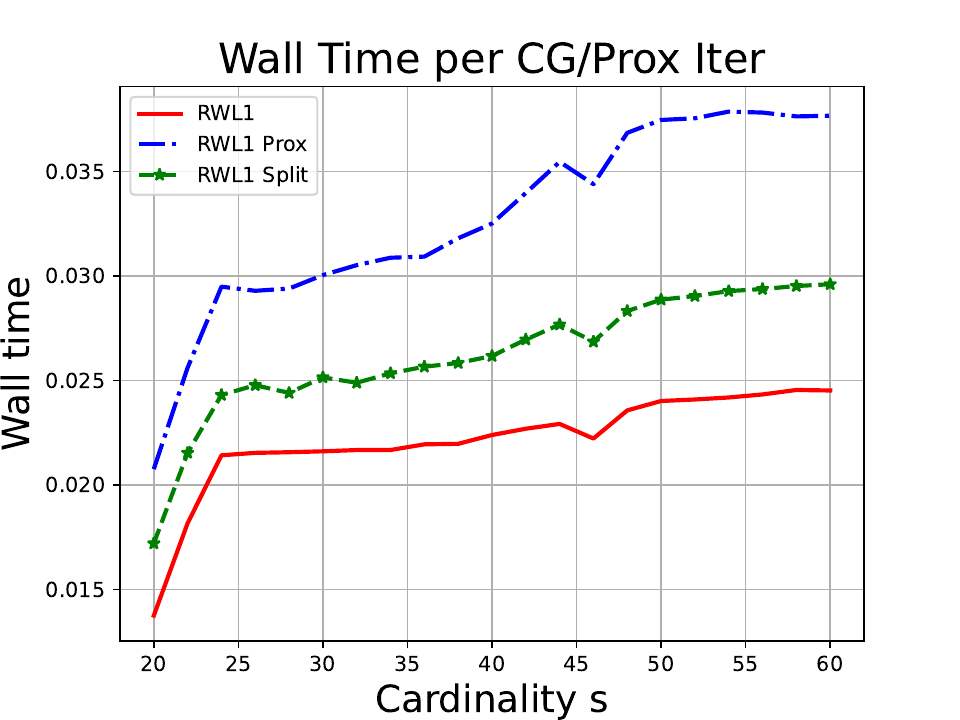}
\caption{Empirical comparison of the RWL1, RWL1 Prox, and RWL1 Split algorithms on sparse signal recovery problems for changing cardinality level}
\label{fig:rwl1_test}
\end{figure}

Figure \ref{fig:rwl1_test} shows that the RWL1 and RWL1 Split algorithms behave almost identically. In fact, out of all the experimental runs, the recovery of these algorithms differs only in 5 examples. Although theoretically, the computational burden of both the RWL1 and the RWL1 Split algorithms is the same, in practice, the RWL1 Split seems slightly slower. This may be caused by the formulation where we directly use $\ell_1$ norm minimization for the RWL1 algorithm, whereas we use the decoupled weighted sum for the RWL1 Split algorithm. Perhaps, internal handling of the direct $\ell_1$ norm is more efficient, hence causing a performance discrepancy. On the other hand, an opposite effect happens when we switch the solver from MOSEK to CLARABEL \cite{Clarabel_2024}; then the RWL1 Split algorithm works faster than the original RWL1 algorithm. Thus, solver capability and the choice of formulation seem to have a notable effect on the practical performance. Furthermore, RWL1 Prox seems to fall behind in convergence rates, while also terminating early. This is likely because the proximal regularization makes the algorithm very conservative, and therefore, the algorithm takes very small step sizes. Also, since the RWL1 Prox algorithm requires a solution of a quadratic program instead of a linear program, the computational burden per subproblem is slightly higher. 

We remark that the RWL1 algorithm, as described in (\ref{eq:rwl1_algorithm}), has a finite last-iterate convergence guarantee since it minimizes a concave function over a polyhedral set with the GFW algorithm \cite{mangasarian1996machine}. This was not pointed out in \cite{candes2008enhancing}. Similarly, finite last-iterate convergence applies to RWL1 Split as described in (\ref{eq:rwl1_star}). However, this result is highly dependent on the polyhedrality of the feasible set $\setX$. If we consider the noisy version of the model (\ref{eq:rwl1}), where we change $Ax^+-Ax^-=b$ with $\|Ax^+-Ax^--b\|_2 \leq \delta$ for some $\delta > 0$, we lose the convergence guarantee for GFW applied to this model. In contrast, since in our modified model (\ref{eq:rwl1_star}) $-f$ satisfies strong convexity and smoothness, our last-iterate convergence results still apply in the noisy case because the function $F$ would still satisfy the KL property. This modification also changes the numerical performance of the algorithms, as we now illustrate.

We use the same numerical setup as before, except we change the construction of the observation vector $b$ and the linear constraints $Ax^+ - Ax^- = b$ as follows: we first generate the noise term $z \in \RR^{100}$ such that $z_i \sim \mathcal{N}(0,10^{-3})$ then set $b = Ax + z$, and change the linear constraints with the norm constraint $\|Ax^+ - Ax^- - b\|_2 \leq \|z\|$. We chose the variance of the noise term to be small so that $\|z\|$ is roughly equal to $0.01$. 

\begin{figure}[ht]
\centering
\includegraphics[width=.325\textwidth]{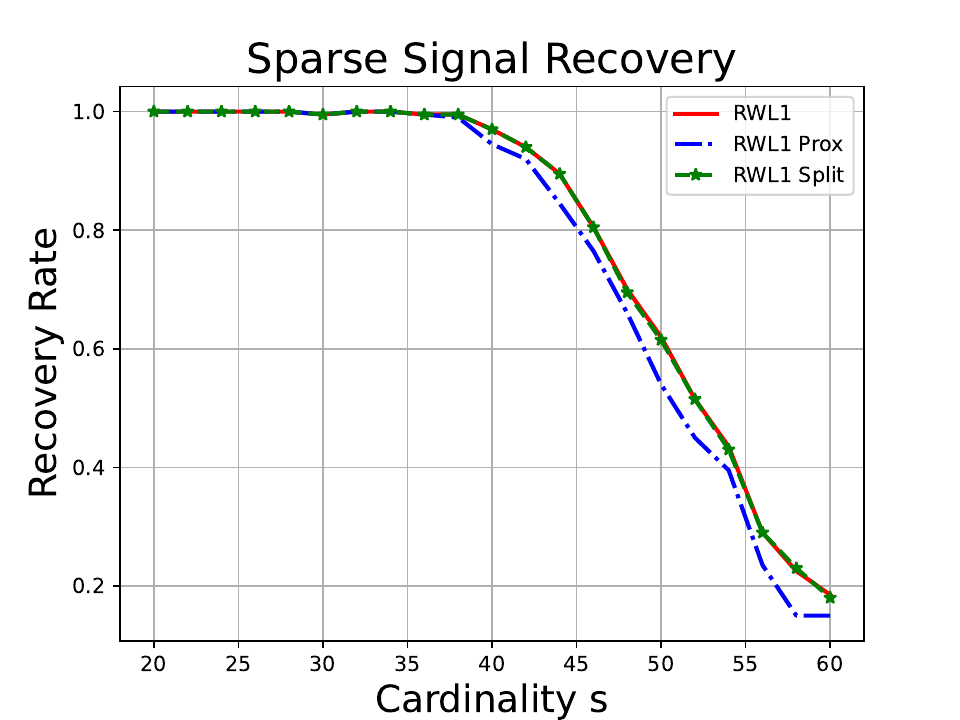}
\includegraphics[width=.325\textwidth]{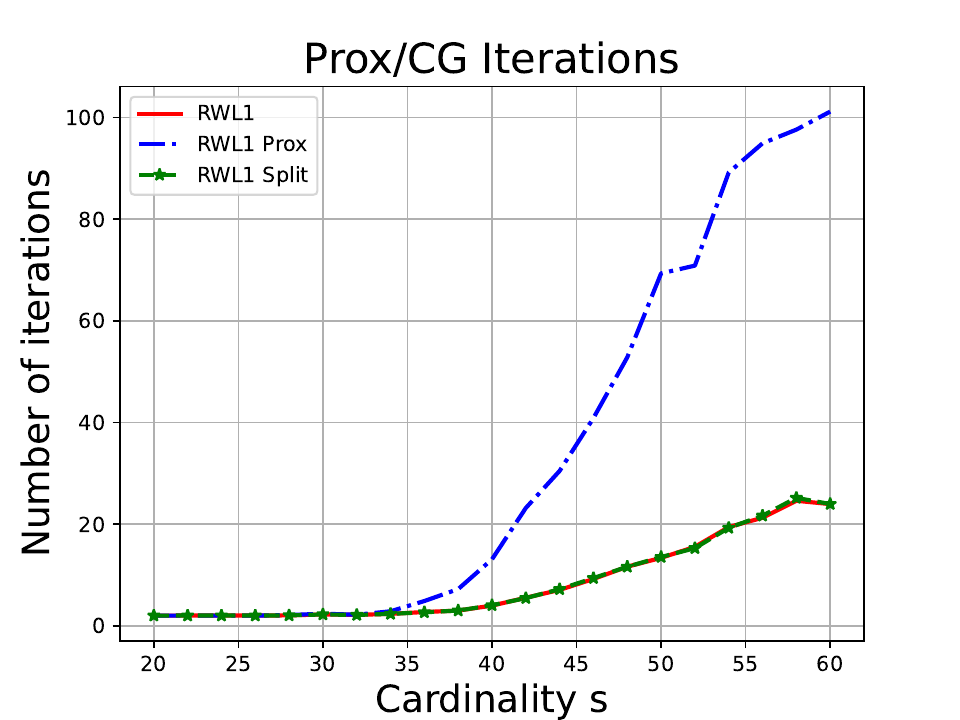}
\includegraphics[width=.325\textwidth]{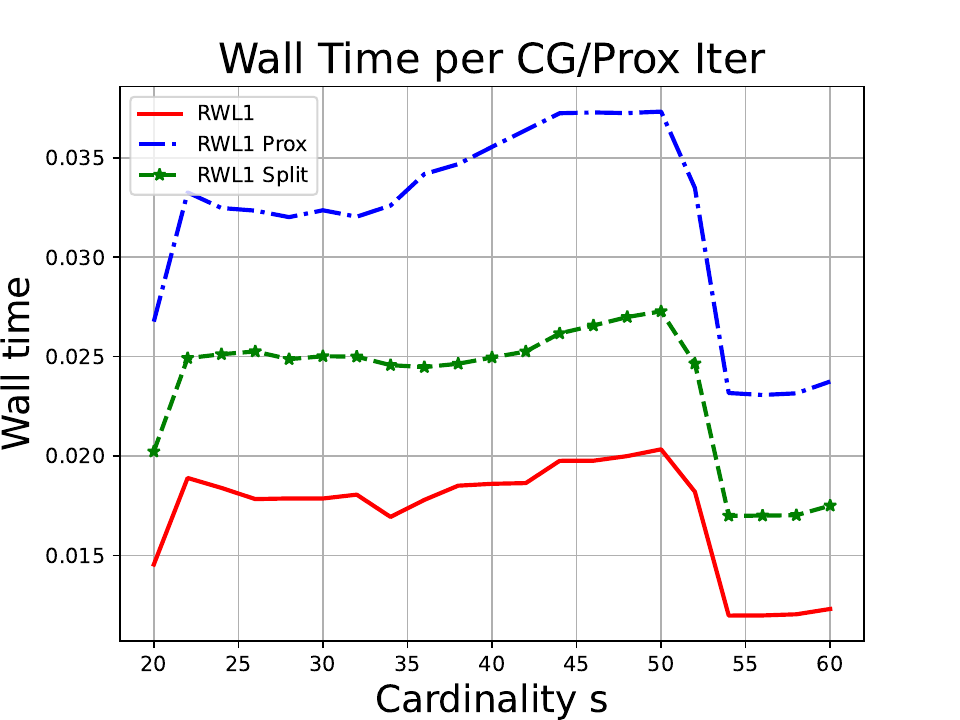}
\caption{Empirical comparison of the RWL1, RWL1 Prox, and RWL1 Split algorithms on noisy sparse signal recovery problems for changing cardinality level}
\label{fig:rwl1_noisy_test}
\end{figure}

Figure \ref{fig:rwl1_noisy_test} shows that the RWL1 Prox algorithm now has a relatively higher recovery rate while still falling behind the RWL1 and the RWL1 Split algorithms. However, the RWL1 Prox algorithm is now much more expensive in terms of the number of proximal iterations. It seems that the RWL1 Prox algorithm progresses very slowly in the polyhedral setting and meets a stopping criterion very fast, hence achieves a lower recovery rate relatively, whereas this phenomenon does not occur in the second case. This experiment numerically shows the superiority of the RWL1 Split algorithm over the RWL1 Prox algorithm in terms of both the computational cost and recovery rates. Additionally, the RWL1 Split and the RWL1 algorithms also demonstrate very similar numerical behavior in this experiment as well.

In summary, we propose a new reweighted $\ell_1$ minimization algorithm, RWL1 Split, that is only slightly different than the original RWL1 algorithm. The RWL1 Split algorithm numerically behaves almost identically to the original RWL1 algorithm, such as in terms of sparse signal recovery, number of CG iterates until meeting a stopping criterion, and the subproblem cost at each iteration. The RWL1 Split algorithm retains its convergence properties when the feasible region is changed. Hence, through a modeling trick, we developed an algorithm that is more flexible in terms of convergence guarantees. More broadly, we can extend this modeling trick to devise theoretically convergent versions of the reweighted trace heuristic (RTH) algorithm for low-rank matrix recovery \cite{mohan2010reweighted} and the PCA Sparsified algorithm for sparsity constrained PCA problem \cite{aktacs2023pca}. Furthermore, we showed that the RWL1 Split algorithm is numerically superior to the RWL1 Prox algorithm, a proximally regularized version of RWL1 devised using the alternating minimization framework.

\subsection{Sparse PCA Lower Bound Approach}

The problem of interest in this section is the following: given an $n$-by-$n$ symmetric matrix, i.e. $A \in \mathbb{S}^{n}$, and an integer $k \in \{1,2,...,n\}$, we seek to find normalized linear factors that maximally correlate with the matrix $A$, while using at most $k$ nonzero components. The mathematical formulation of this problem leads to the following model

\begin{equation}
\label{eq:sparse_pca}
\begin{array}{ll@{}ll}
\underset{x}\max \quad & x^TAx  \\
\text{st.} \quad & \|x\|_2 = 1, \\
& \|x\|_0 \leq k. \\
\end{array}
\end{equation}

This problem is known as sparse principal component analysis (SPCA). This problem is a difficult nonconvex problem since it consists of maximizing a convex quadratic function over a compact set. However, this problem fits our framework perfectly. To solve this problem, a simple gradient scheme is proposed in \cite{luss2013conditional} that maximizes the gradient lower bound. 

\begin{equation}
\label{eq:gpower_algorithm}
\begin{array}{ll@{}ll}
x_{n+1} \in \underset{x}\argmax \left\{x_n^TAx: \|x\|_2 = 1, \|x\|_0 \leq k \right\}.\\
\end{array}
\end{equation}

Convergence analysis of this algorithm was mostly limited to Theorem \ref{thm:classic_convergence}, where the convergence of the point sequence $\{x_k\}_{k\in \NN}$ was unclear. However, we can apply Theorem \ref{thm:global_convergence_gfw} by modifying the problem slightly; this modification is well-known and is also used e.g. by \cite{luss2013conditional}. Since we have a spherical constraint ($\|x\|_2 = 1$), we can instead solve the problem

\begin{equation}
\label{eq:sparse_pca_strongly_convex}
\begin{array}{ll@{}ll}
\underset{x}\max \quad & x^T(A + \sigma I_n)x  \\
\text{st.} \quad & \|x\|_2 = 1, \\
& \|x\|_0 \leq k. \\
\end{array}
\end{equation}

where $\sigma$ is some positive constant and $I_n$ is the $n$ by $n$ identity matrix. Clearly the problems (\ref{eq:sparse_pca}) and (\ref{eq:sparse_pca_strongly_convex}) admit an identical set of optimal solutions. The latter reformulation allows us to assume strong convexity on $-f$ without loss of generality. Also, in both formulations, $f$ has a constant Hessian, so it satisfies the smoothness criterion. Then, defining the quantities $-f = x^T(A + \sigma I_n)x$, $\setX = \{x \in \RR^n:\|x\|_2 = 1, \|x\|_0 \leq k\}$, and $F(x) = f(x) + \I_{\setX}(x)$, we can recognize that the function $F$ is semialgebraic. Therefore, both the global convergence result given by Theorem \ref{thm:global_convergence_gfw} and the convergence rate result described by Theorem \ref{thm:convergence_rate_gfw} apply. We do not report numerical results for this application as the algorithm is known to work well in practice \cite{luss2013conditional}; instead, we show a new convergence result. 

We also note that in \cite{luss2013conditional}, alternatives to the model (\ref{eq:sparse_pca}) are discussed. One such approach is the relaxation of $\ell_0$ norm constraint to a $\ell_1$ norm constraint, and then the GFW algorithm is applied. The Theorems \ref{thm:global_convergence_gfw} and \ref{thm:convergence_rate_gfw} apply for that instance as well. 

\subsection{Simple Parallel Algorithm for the Max-Cut SDP Formulation}

Semidefinite Programming (SDP) is a popular tool for approximating combinatorial problems \cite{alizadeh1995interior,lovasz1991cones}.  We focus on the following general SDP template, which arises as a convex relaxation to the Max-Cut problem~\cite{goemans1995improved}, graphical model inference~\cite{erdogdu2017inference}, community detection~\cite{bandeira2016low}, and group synchronization \cite{mei2017solving}. 

\begin{equation}
\label{eq:max_cut_sdp}
\begin{array}{lllll}
\underset{Z}\max & \langle A, Z\rangle  \\
\text{st.} & Z_{ii} = 1 \quad \forall i = 1,...,n, \\
& Z  \succeq 0,
\end{array}
\end{equation}

where $A,Z \in \mathbb{S}^{n}$ are symmetric matrices of size $n$ by $n$. Although SDP models have desirable properties theoretically, computing an optimal solution numerically still remains computationally daunting. For instance, interior point methods have an $O(n^6)$ complexity per iteration and large memory requirements. Using interior point methods quickly becomes infeasible for large-scale problems. To alleviate this problem, a common approach is to introduce a low-rank factorization of the variable $Z = BB^T$ where $B \in \RR^{n \times r}$. In terms of the new variable, the model becomes

\begin{equation}
\label{eq:max_cut_bm}
\begin{array}{lllll}
\underset{Z}\max & \langle A, BB^T\rangle  \\
\text{st.} & \|b_{i}\|_2 = 1 \quad \forall i = 1,...,n,
\end{array}
\end{equation}

where $b_i$ denotes the $i$th row i.e. $B = [b_1,b_2,...,b_n]^T$. The advantage of introducing the variable $B$ is that the positive semidefinite cone constraint is removed, and if $r$ is chosen smaller than $n$, both the computations and the storage costs are cheaper. On the other hand, the resulting problem is now nonconvex. This method is called Burer-Monteiro Factorization \cite{burer2003nonlinear}, and this approach was combined with an augmented Lagrangian method to solve an SDP in the standard form. 

For the special SDP template in (\ref{eq:max_cut_bm}), approaches such as block-coordinate maximization \cite{javanmard2016phase,wang2017mixing,erdogdu2022convergence}, Riemannian gradient (RGD) \cite{javanmard2016phase,mei2017solving}, and Riemannnian trust-region (RTR) methods \cite{absil2007trust,boumal2016non,journee2010low} demonstrate better performance than the originally proposed augmented Lagrangian method in practice. Consider the block-coordinate maximization method (BCM) described in Algorithm \ref{alg:max_cut_bm_bcm}.

\begin{algorithm}[ht]
\caption{BCM algorithm for model (\ref{eq:max_cut_bm})}
\label{alg:max_cut_bm_bcm}
\begin{algorithmic}[1]
\STATE{\textbf{Initialize:} $b_i^0 \in \RR^{r} \ \forall i = 1,...,n$ }
\FOR{$k = 0,1,2, ...$}
\FOR{$i = 1,2, ...,n$}
\STATE{$g^{k+1}_{i} = \sum\limits_{j=1}^{i-1} a_{ij}b_j^{k+1} + \sum\limits_{j=i+1}^n a_{ij}b_j^k $}
\STATE{$b^{k+1}_{i} = \frac{g^{k+1}_{i}}{\|g^{k+1}_{i}\|_2} $}
\ENDFOR
\ENDFOR
\end{algorithmic}
\end{algorithm}

The main idea behind the BCM algorithm is to locally maximize the objective over each block $i$. Instead of BCM, we propose the GFW algorithm for this model, which gives Algorithm \ref{alg:max_cut_bm_cg}. We note that the GFW algorithm operates on a slightly different (shifted) model to apply the convergence results, as we explain in the upcoming discussion.

\begin{algorithm}[ht]
\caption{GFW algorithm for model (\ref{eq:max_cut_bm_shifted})}
\label{alg:max_cut_bm_cg}
\begin{algorithmic}[1]
\STATE{\textbf{Initialize:} $B \in \RR^{n \times r}$ }
\FOR{$k = 0,1,2, ...$}
\STATE{$G^{k+1} = AB^k $}
\STATE{$D_{ii}^{k+1} = \frac{1}{\|G^{k+1}e_i\|_2} $}
\STATE{$B^{k+1} = D^{k+1}G^{k+1} $}
\ENDFOR
\end{algorithmic}
\end{algorithm}

Although it is not clear from the presentation of the algorithms, the BCM and GFW algorithms are very similar. The key difference is that the BCM algorithm updates each block one by one sequentially, whereas the GFW algorithm updates all blocks at once. In other words, the GFW algorithm parallelizes the block updates.

Since the original model \ref{eq:max_cut_sdp} is diagonally constrained, we can change the diagonal elements of $A$ without changing the set of optimal points. Thus, we can work on the following model without loss of generality

\begin{equation}
\label{eq:max_cut_bm_shifted}
\begin{array}{lllll}
\underset{Z}\max & \langle A + \sigma I_n, BB^T\rangle  \\
\text{st.} & \|b_{i}\|_2 = 1 \quad \forall i = 1,...,n.
\end{array}
\end{equation}

Since we can work with $A + \sigma I_n$ for any $\sigma > 0$, we can assume strong convexity in terms of $b_i$'s by choosing a suitable $\sigma$. Let $b$ denote the flattened version of $B$ i.e. $b = [b_1^T,b_2^T,...,b_n^T]$. Then the Hessian of $f$ with respect to $b$ is $A \otimes I_n$ where $\otimes$ denotes the Kronecker product. The matrix $A \otimes I_n$ has the same eigenvalues as $A$, with multiplicity $n$ for each eigenvalue \cite{horn1994topics}. Thus, setting $\sigma = -\lambda_{\min}(A) + \gamma$ for some $\gamma > 0$ is an appropriate choice for the algorithm where $\lambda_{\min}(A)$ denotes the minimum eigenvalue of $A$. Thus, the objective in (\ref{eq:max_cut_bm_shifted}) is strongly convex for suitably chosen $\sigma$ and smooth since the Hessian is constant. Let us define the following quantities: $-f = \langle A + \sigma I_n, BB^T \rangle $, $\setX = \{B \in \RR^{n\times r}: \|b_i\|_2 = 1 \, \forall i = 1,...,n\}$ and $F(x) = f(x) + I_{\setX}(x)$. The feasible region is compact since it is the Cartesian product of spheres. Hence, Assumption \ref{as:convex_smooth_compact} is satisfied.  Furthermore, we can recognize that $F$ is semialgebraic since $f(x)$ is a real polynomial and the feasible region $\setX$ is a semialgebraic set. Therefore, the global convergence result given by Theorem \ref{thm:global_convergence_gfw} and the rate of convergence result given by Theorem \ref{thm:convergence_rate_gfw} apply.

Notice that although Algorithm~\ref{alg:max_cut_bm_cg} is guaranteed to converge to a first-order stationary point by the global convergence theorem, it may not recover the global optimum. Nevertheless, for sufficiently large $k$, computing a local minimum is enough to recover the optimal solution \cite{burer2005local}. This statement was sharpened to second-order critical points instead of a local minimum for almost all $A$ in \cite{boumal2016non}. Furthermore, local minimums are shown to be high-quality estimates of the global minimum in \cite{mei2017solving}. These properties motivate us to use Algorithm \ref{alg:max_cut_bm_cg}.

Additionally, since the maximizer of a linear function over the sphere is unique, it can be shown that the sequence generated by Algorithm \ref{alg:max_cut_bm_cg} will converge to a first-order strict stationary point. In other words, the following holds;

\begin{equation}
\label{eq:max_cut_bm_stationary}
\begin{array}{lllll}
\langle AB^*, B-B^* \rangle < 0 \quad \forall B \in \setX, B \neq B^*
\end{array}
\end{equation}

where $B^*$ is the limit point of the sequence generated by \ref{alg:max_cut_bm_cg}. Remark that this also shows that $B^*$ satisfies the second-order necessary conditions since the tangent space is empty \cite{bertsekas1997nonlinear}. Nevertheless, this observation does not allow us to invoke the global optimality proved in \cite{boumal2016non} since their characterization requires the Riemannian Hessian rather than the Euclidean Hessian. 

We now present computational results illustrating the empirical performance of the GFW algorithm. GFW and BCM algorithms are implemented on MATLAB. We also implemented RGD and RTR algorithms using the Manopt package \cite{boumal2014manopt}. In addition, we implemented an algorithm that starts with the GFW algorithm and switches to RTR when the magnitude of the Riemannian gradient is small, as the second-order information is useful near the limit point. This was motivated and proposed in \cite{erdogdu2022convergence}. 

In the experiments, we generate symmetric matrices $A = (G+ G^T)/n$ where each $G_{ij}$ is sampled from a standard normal distribution for $i,j \in \{1,2,..,n\}$ with the problem size $n= 20,000$. We initialize $B \in \RR^{n \times r}$ randomly on the Cartesian product of spheres, set $r = \lceil{\sqrt{2n}}\rceil$, and fix $\sigma = 25\times10^{-4}$. We set the maximum wall time as $60$ seconds for all algorithms and repeat the experiment $50$ times.

In the same setup as described in the previous paragraph, we also demonstrate the effect of the choice of the shifting parameter $\sigma$ to guarantee strong convexity. For that purpose, we try $4$ different values, $\sigma_1 = 0$, $\sigma_2 = 10^{-3}$, $\sigma_3 = 25\times10^{-4}$ and $\sigma_4 = -\lambda_{min}(A) + 0.1$. We note that the first three choices of $\sigma$ do not guarantee strong convexity. 

\begin{figure}[ht]
\centering
\includegraphics[width=.49\textwidth]{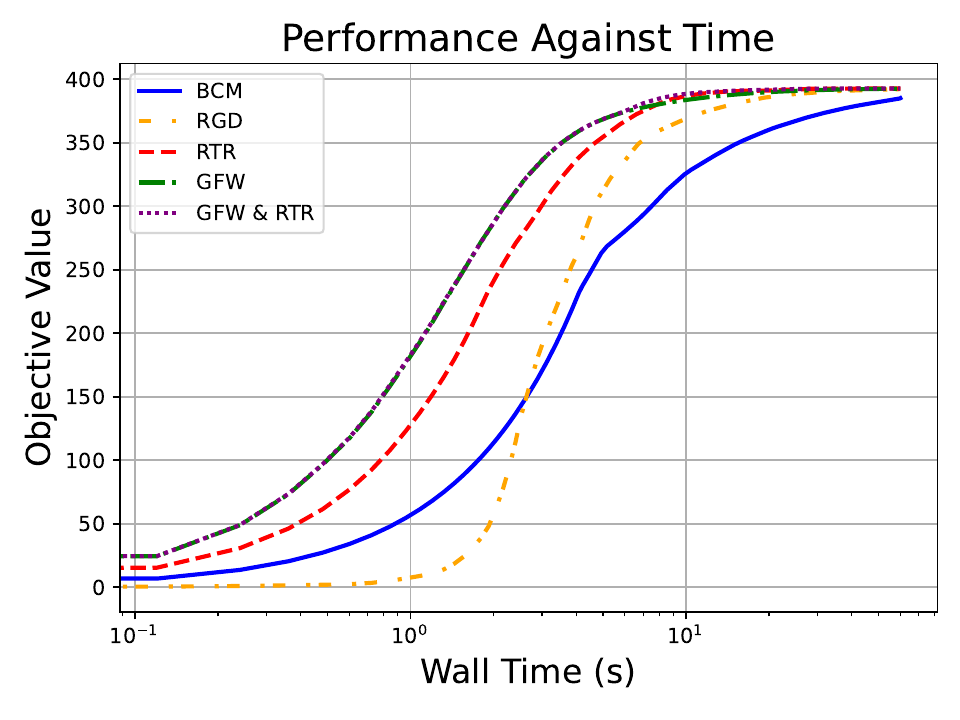}
\includegraphics[width=.49\textwidth]{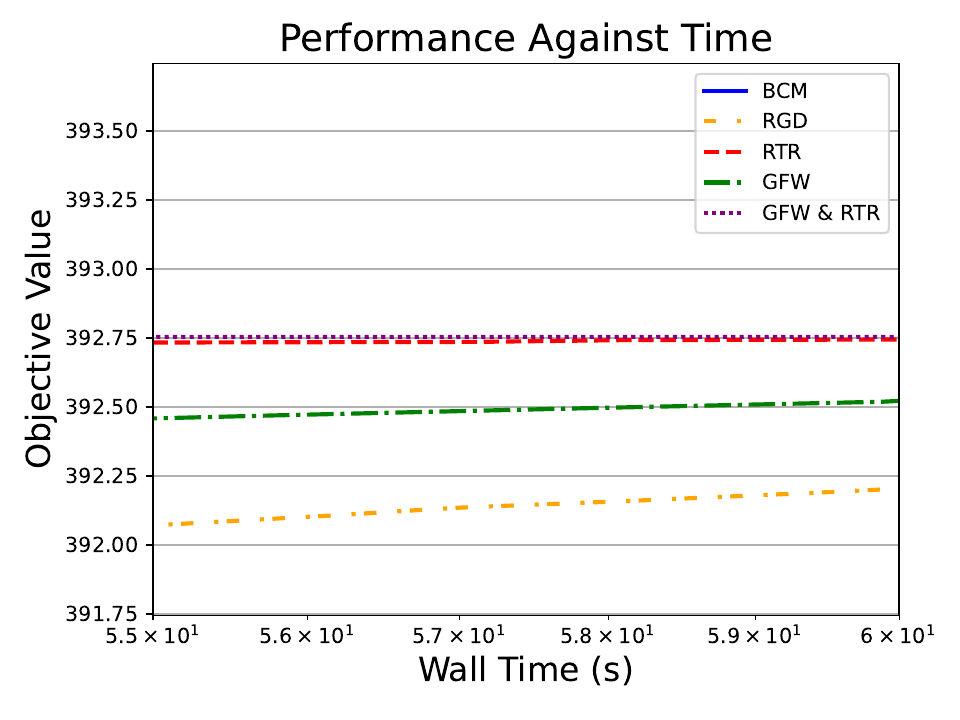}
\caption{Empirical comparison of various algorithms for solving the Max-Cut SDP Relaxation. The figure on the left shows the objective value against the wall time. GFW and GFW \& RTR algorithms overlap until $t \approx 7$ seconds. The figure on the right provides a zoomed-in view, focusing on the final $5$ seconds.}
\label{fig:max_cut_test}
\end{figure}

Figure \ref{fig:max_cut_test} shows that the GFW algorithm is much faster than the BCM algorithm. The final objective value achieved by the BCM algorithm is surpassed by the GFW algorithm within 10 seconds. Similarly, the GFW algorithm beats the RGD algorithm as well. While we observe early fast convergence for the GFW algorithm, the RTR algorithm seems to catch up and even surpass it with a small margin. This is likely because of the accurate second-order information used by the RTR algorithm, which allows it to discover highly accurate solutions near the limit point. The combined algorithm GFW \& RTR enjoys both the early fast convergence and later accurate solutions. Thus, the plots suggest that the best algorithm is the combined GFW \& RTR algorithm.

The previous experiments were performed on an Intel i7-13700K CPU. Next we investigate the numerical performance when GFW and RTR are run on GPUs, via Google Colab on an NVIDIA A100 GPU. We have implemented GFW using CuPy \cite{cupy_learningsys2017} and RTR using JAX \cite{jax2018github} via Pymanopt \cite{pymanopt2016}. We also increased the problem size to $n=30,000$ and set $\sigma = 50/n$. The rest of the setup is the same. 
Because of the complexity of measuring runtime performance on GPU, our numerical measurements for the GPU experiments are much coarser. 
Specifically, since the computations are done asynchronously on the GPU, enforcing a strict time limit for the number of computations is not possible to do reliably. 
Moreover, Pymanopt does not report the cost function throughout the iterations or the time spent between iterations.
Because of these issues, we roughly estimate the execution time and read the objective value at $t\approx 10$, $t\approx 20$, and $t\approx 40$ seconds for the GFW algorithm. 
For RTR, we give a time limit of $t = 10$, $t = 30$, and $t = 60$ seconds and check the actual time spent as well as the objective value at that time (we give RTR additional time because even with this additional time, GFW performs better, and RTR often exceeds its time limit, which is why we report e.g. $42.1$ seconds for $t=30$). We report the mean value of the time spent and the objective value computed by the algorithms.

\begin{table}[ht]
\centering
\begin{tabular}{lcccccl}\toprule
& \multicolumn{2}{c}{RTR} & \multicolumn{2}{c}{GFW}
\\ \cmidrule(lr){2-3} \cmidrule(lr){4-5}
           & $f$            & Time (s)         & $f$        &  Time (s)    \\ \midrule
Reading 1  & 530.6896  & 12.2 & 532.3243  & 10.3 \\
Reading 2 & 532.3255  & 42.1 & 532.3379  & 20.6 \\
Reading 3 & 532.3384  & 78.2 & 532.3399 & 41.3 \\ \bottomrule
\end{tabular}
\caption{Empirical comparison RTR and GFW for solving the Max-Cut SDP Relaxation on a GPU}
\label{table:gpu}
\end{table}

Table \ref{table:gpu} shows that when the algorithms are implemented on a GPU, the plain GFW algorithm is much faster than RTR; GFW finds a highly accurate solution after ten seconds, whereas RTR takes about 40 seconds to do so. This can be attributed to the fact that the GFW algorithm is well-suited for parallel computing.

\begin{figure}[ht]
\centering
\includegraphics[width=.49\textwidth]{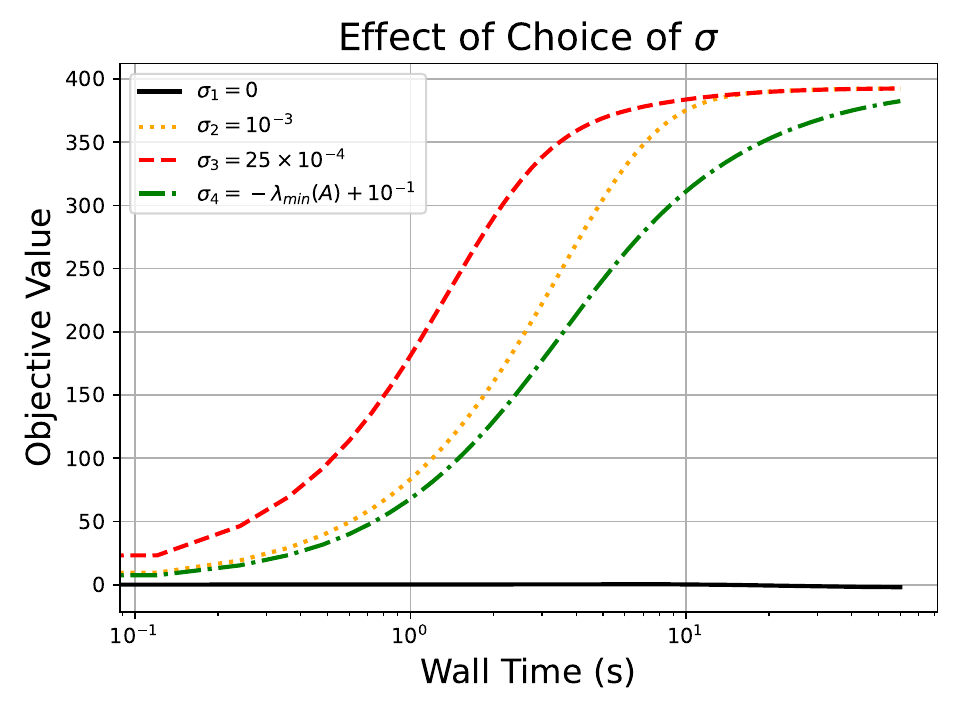}
\includegraphics[width=.49\textwidth]{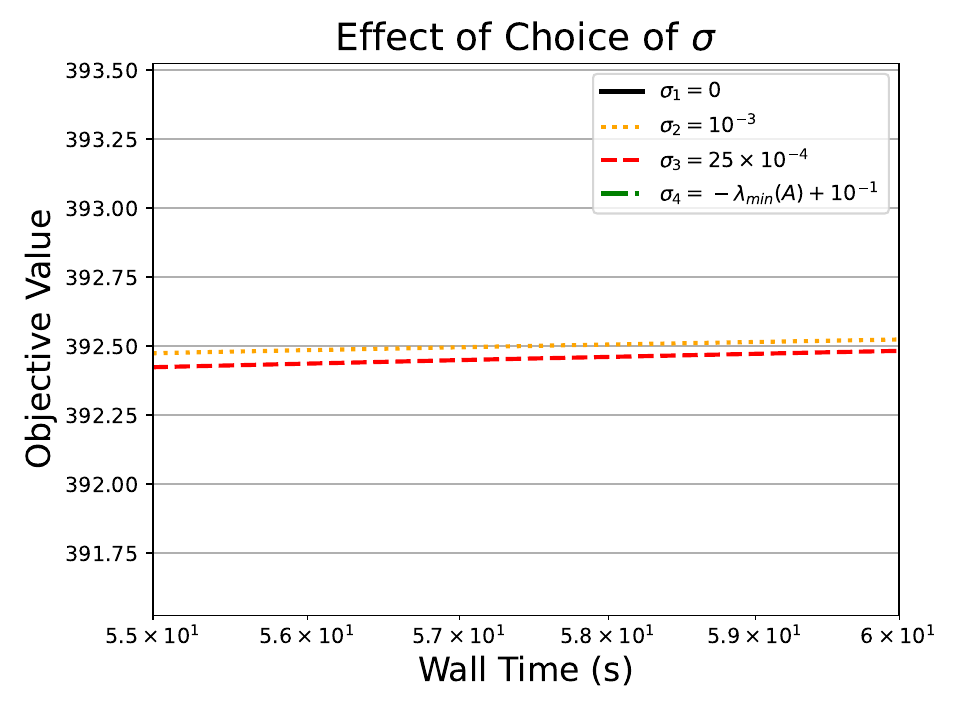}
\caption{Empirical performance of the GFW algorithm for varying shifting parameter $\sigma$ on the Max-Cut SDP relaxation. The figure on the left shows the objective value against the wall time. The figure on the right provides a zoomed-in view, focusing on the final $5$ seconds.}
\label{fig:max_cut_test2}
\end{figure}

Finally, Figure \ref{fig:max_cut_test2} investigates the importance of the choice of $\sigma$. For the value of zero, the algorithm does not work. For small but large enough values, the algorithm enjoys a quick convergence, while setting $\sigma$ very large seems to slow down the convergence. Finally, we remark that, although it seems like $\sigma_3$ is the best choice in terms of objective value convergence, $\sigma_2$ is ahead by a tiny margin towards the end.

\section{Conclusion}
In this paper, we proved new convergence results for the GFW algorithm applied to strongly convex maximization problems under mild assumptions. Based on the new convergence result, we show additional convergence properties for a sparse principal component analysis algorithm,  derive a theoretically convergent reweighted $\ell_1$ algorithm, and propose a new algorithm for the semidefinite relaxation of the Max-Cut problem. It would be of interest to sharpen the convergence results to a more precise form by deriving explicit bounds for the constants, which we leave as further research. 

\section*{Acknowledgments}
We are grateful to Santiago Balseiro for their helpful discussions and feedback in the initial stages of the paper.

\bibliographystyle{ieeetr}   
\bibliography{references}

\begin{thebibliography}{10}

\bibitem{zwart1974global}
P.~B. Zwart, ``Global maximization of a convex function with linear inequality constraints,'' {\em Operations Research}, vol.~22, no.~3, pp.~602--609, 1974.

\bibitem{mangasarian1996machine}
O.~L. Mangasarian, ``Machine learning via polyhedral concave minimization,'' in {\em Applied Mathematics and Parallel Computing: Festschrift for Klaus Ritter}, pp.~175--188, Springer, 1996.

\bibitem{d2004direct}
A.~d'Aspremont, L.~Ghaoui, M.~Jordan, and G.~Lanckriet, ``A direct formulation for sparse pca using semidefinite programming,'' {\em Advances in neural information processing systems}, vol.~17, 2004.

\bibitem{zass2006nonnegative}
R.~Zass and A.~Shashua, ``Nonnegative sparse pca,'' {\em Advances in neural information processing systems}, vol.~19, 2006.

\bibitem{candes2008enhancing}
E.~J. Candes, M.~B. Wakin, and S.~P. Boyd, ``Enhancing sparsity by reweighted $\ell_1$ minimization,'' {\em Journal of Fourier analysis and applications}, vol.~14, pp.~877--905, 2008.

\bibitem{mohan2010reweighted}
K.~Mohan and M.~Fazel, ``Reweighted nuclear norm minimization with application to system identification,'' in {\em Proceedings of the 2010 American Control Conference}, pp.~2953--2959, IEEE, 2010.

\bibitem{aktacs2023pca}
F.~S. Akta{\c{s}} and M.~{\c{C}}.~Pinar, ``Pca sparsified,'' {\em SIAM Journal on Optimization}, vol.~33, no.~3, pp.~2089--2117, 2023.

\bibitem{goemans1995improved}
M.~X. Goemans and D.~P. Williamson, ``Improved approximation algorithms for maximum cut and satisfiability problems using semidefinite programming,'' {\em Journal of the ACM (JACM)}, vol.~42, no.~6, pp.~1115--1145, 1995.

\bibitem{benson1995concave}
H.~P. Benson, ``Concave minimization: theory, applications and algorithms,'' in {\em Handbook of global optimization}, pp.~43--148, Springer, 1995.

\bibitem{selvi2022convex}
A.~Selvi, A.~Ben-Tal, R.~Brekelmans, and D.~den Hertog, ``Convex maximization via adjustable robust optimization,'' {\em INFORMS Journal on Computing}, vol.~34, no.~4, pp.~2091--2105, 2022.

\bibitem{ben2022algorithm}
A.~Ben-Tal and E.~Roos, ``An algorithm for maximizing a convex function based on its minimum,'' {\em INFORMS journal on computing}, vol.~34, no.~6, pp.~3200--3214, 2022.

\bibitem{lipp2016variations}
T.~Lipp and S.~Boyd, ``Variations and extension of the convex--concave procedure,'' {\em Optimization and Engineering}, vol.~17, pp.~263--287, 2016.

\bibitem{yurtsever2022cccp}
A.~Yurtsever and S.~Sra, ``Cccp is frank-wolfe in disguise,'' {\em Advances in Neural Information Processing Systems}, vol.~35, pp.~35352--35364, 2022.

\bibitem{pardalos1988checking}
P.~M. Pardalos and G.~Schnitger, ``Checking local optimality in constrained quadratic programming is np-hard,'' {\em Operations Research Letters}, vol.~7, no.~1, pp.~33--35, 1988.

\bibitem{pardalos1986methods}
P.~M. Pardalos and J.~B. Rosen, ``Methods for global concave minimization: A bibliographic survey,'' {\em Siam Review}, vol.~28, no.~3, pp.~367--379, 1986.

\bibitem{audet2005essays}
C.~Audet, P.~Hansen, and G.~Savard, {\em Essays and surveys in global optimization}, vol.~7.
\newblock Springer Science \& Business Media, 2005.

\bibitem{andrianova2016one}
A.~Andrianova, A.~Korepanova, and I.~Halilova, ``One algorithm for branch and bound method for solving concave optimization problem,'' in {\em IOP Conference Series: Materials Science and Engineering}, vol.~158, p.~012005, IOP Publishing, 2016.

\bibitem{luss2013conditional}
R.~Luss and M.~Teboulle, ``Conditional gradient algorithmsfor rank-one matrix approximations with a sparsity constraint,'' {\em siam REVIEW}, vol.~55, no.~1, pp.~65--98, 2013.

\bibitem{bertsekas1997nonlinear}
D.~P. Bertsekas, ``Nonlinear programming,'' {\em Journal of the Operational Research Society}, vol.~48, no.~3, pp.~334--334, 1997.

\bibitem{frank1956algorithm}
M.~Frank, P.~Wolfe, {\em et~al.}, ``An algorithm for quadratic programming,'' {\em Naval research logistics quarterly}, vol.~3, no.~1-2, pp.~95--110, 1956.

\bibitem{levitin1966constrained}
E.~S. Levitin and B.~T. Polyak, ``Constrained minimization methods,'' {\em USSR Computational mathematics and mathematical physics}, vol.~6, no.~5, pp.~1--50, 1966.

\bibitem{dunn1980convergence}
J.~C. Dunn, ``Convergence rates for conditional gradient sequences generated by implicit step length rules,'' {\em SIAM Journal on Control and Optimization}, vol.~18, no.~5, pp.~473--487, 1980.

\bibitem{hazan2012projection}
E.~Hazan and S.~Kale, ``Projection-free online learning,'' {\em arXiv preprint arXiv:1206.4657}, 2012.

\bibitem{jaggi2013revisiting}
M.~Jaggi, ``Revisiting frank-wolfe: Projection-free sparse convex optimization,'' in {\em International conference on machine learning}, pp.~427--435, PMLR, 2013.

\bibitem{yurtsever2018conditional}
A.~Yurtsever, O.~Fercoq, F.~Locatello, and V.~Cevher, ``A conditional gradient framework for composite convex minimization with applications to semidefinite programming,'' in {\em International Conference on Machine Learning}, pp.~5727--5736, PMLR, 2018.

\bibitem{yurtsever2019conditional}
A.~Yurtsever, O.~Fercoq, and V.~Cevher, ``A conditional-gradient-based augmented lagrangian framework,'' in {\em International Conference on Machine Learning}, pp.~7272--7281, PMLR, 2019.

\bibitem{kerdreux2020accelerating}
T.~Kerdreux, {\em Accelerating conditional gradient methods}.
\newblock PhD thesis, Universit{\'e} Paris sciences et lettres, 2020.

\bibitem{journee2010generalized}
M.~Journ\'{e}e, Y.~Nesterov, P.~Richt\'{a}rik, and R.~Sepulchre, ``Generalized power method for sparse principal component analysis,'' {\em Journal of Machine Learning Research}, vol.~11, p.~517–553, Mar. 2010.

\bibitem{chaudhury2024competitiveequilibriumchoresdual}
B.~R. Chaudhury, C.~Kroer, R.~Mehta, and T.~Nan, ``Competitive equilibrium for chores: from dual eisenberg-gale to a fast, greedy, lp-based algorithm,'' 2024.

\bibitem{bolte2018first}
J.~Bolte, S.~Sabach, M.~Teboulle, and Y.~Vaisbourd, ``First order methods beyond convexity and lipschitz gradient continuity with applications to quadratic inverse problems,'' {\em SIAM Journal on Optimization}, vol.~28, no.~3, pp.~2131--2151, 2018.

\bibitem{teboulle2020novel}
M.~Teboulle and Y.~Vaisbourd, ``Novel proximal gradient methods for nonnegative matrix factorization with sparsity constraints,'' {\em SIAM Journal on Imaging Sciences}, vol.~13, no.~1, pp.~381--421, 2020.

\bibitem{attouch2010proximal}
H.~Attouch, J.~Bolte, P.~Redont, and A.~Soubeyran, ``Proximal alternating minimization and projection methods for nonconvex problems: An approach based on the kurdyka-{\l}ojasiewicz inequality,'' {\em Mathematics of operations research}, vol.~35, no.~2, pp.~438--457, 2010.

\bibitem{bolte2014proximal}
J.~Bolte, S.~Sabach, and M.~Teboulle, ``Proximal alternating linearized minimization for nonconvex and nonsmooth problems,'' {\em Mathematical Programming}, vol.~146, no.~1, pp.~459--494, 2014.

\bibitem{lacoste2016convergence}
S.~Lacoste-Julien, ``Convergence rate of frank-wolfe for non-convex objectives,'' {\em arXiv preprint arXiv:1607.00345}, 2016.

\bibitem{yurtsever2017sketchy}
A.~Yurtsever, M.~Udell, J.~Tropp, and V.~Cevher, ``Sketchy decisions: Convex low-rank matrix optimization with optimal storage,'' in {\em Artificial intelligence and statistics}, pp.~1188--1196, PMLR, 2017.

\bibitem{rockafellar1970convex}
R.~T. Rockafellar, {\em Convex Analysis}.
\newblock Princeton University Press, 1970.

\bibitem{leplat2023conic}
V.~Leplat, Y.~Nesterov, N.~Gillis, and F.~Glineur, ``Conic optimization-based algorithms for nonnegative matrix factorization,'' {\em Optimization Methods and Software}, vol.~38, no.~4, pp.~837--859, 2023.

\bibitem{journee2010low}
M.~Journ{\'e}e, F.~Bach, P.-A. Absil, and R.~Sepulchre, ``Low-rank optimization on the cone of positive semidefinite matrices,'' {\em SIAM Journal on Optimization}, vol.~20, no.~5, pp.~2327--2351, 2010.

\bibitem{rockafellar2009variational}
R.~T. Rockafellar and R.~J.-B. Wets, {\em Variational analysis}, vol.~317.
\newblock Springer Science \& Business Media, 2009.

\bibitem{attouch2009convergence}
H.~Attouch and J.~Bolte, ``On the convergence of the proximal algorithm for nonsmooth functions involving analytic features,'' {\em Mathematical Programming}, vol.~116, pp.~5--16, 2009.

\bibitem{bauschke2017descent}
H.~H. Bauschke, J.~Bolte, and M.~Teboulle, ``A descent lemma beyond lipschitz gradient continuity: first-order methods revisited and applications,'' {\em Mathematics of Operations Research}, vol.~42, no.~2, pp.~330--348, 2017.

\bibitem{bolte2007lojasiewicz}
J.~Bolte, A.~Daniilidis, and A.~Lewis, ``The {\l}ojasiewicz inequality for nonsmooth subanalytic functions with applications to subgradient dynamical systems,'' {\em SIAM Journal on Optimization}, vol.~17, no.~4, pp.~1205--1223, 2007.

\bibitem{bolte2007clarke}
J.~Bolte, A.~Daniilidis, A.~Lewis, and M.~Shiota, ``Clarke subgradients of stratifiable functions,'' {\em SIAM Journal on Optimization}, vol.~18, no.~2, pp.~556--572, 2007.

\bibitem{attouch2013convergence}
H.~Attouch, J.~Bolte, and B.~F. Svaiter, ``Convergence of descent methods for semi-algebraic and tame problems: proximal algorithms, forward--backward splitting, and regularized gauss--seidel methods,'' {\em Mathematical Programming}, vol.~137, no.~1, pp.~91--129, 2013.

\bibitem{donoho2006compressed}
D.~L. Donoho, ``Compressed sensing,'' {\em IEEE Transactions on information theory}, vol.~52, no.~4, pp.~1289--1306, 2006.

\bibitem{bruckstein2009sparse}
A.~M. Bruckstein, D.~L. Donoho, and M.~Elad, ``From sparse solutions of systems of equations to sparse modeling of signals and images,'' {\em SIAM review}, vol.~51, no.~1, pp.~34--81, 2009.

\bibitem{mosek}
M.~ApS, {\em MOSEK Optimizer API for Python 11.0.12}, 2025.

\bibitem{diamond2016cvxpy}
S.~Diamond and S.~Boyd, ``{CVXPY}: {A} {P}ython-embedded modeling language for convex optimization,'' {\em Journal of Machine Learning Research}, vol.~17, no.~83, pp.~1--5, 2016.

\bibitem{Clarabel_2024}
P.~J. Goulart and Y.~Chen, ``Clarabel: An interior-point solver for conic programs with quadratic objectives,'' 2024.

\bibitem{alizadeh1995interior}
F.~Alizadeh, ``Interior point methods in semidefinite programming with applications to combinatorial optimization,'' {\em SIAM journal on Optimization}, vol.~5, no.~1, pp.~13--51, 1995.

\bibitem{lovasz1991cones}
L.~Lov{\'a}sz and A.~Schrijver, ``Cones of matrices and set-functions and 0--1 optimization,'' {\em SIAM journal on optimization}, vol.~1, no.~2, pp.~166--190, 1991.

\bibitem{erdogdu2017inference}
M.~A. Erdogdu, Y.~Deshpande, and A.~Montanari, ``Inference in graphical models via semidefinite programming hierarchies,'' {\em Advances in Neural Information Processing Systems}, vol.~30, 2017.

\bibitem{bandeira2016low}
A.~S. Bandeira, N.~Boumal, and V.~Voroninski, ``On the low-rank approach for semidefinite programs arising in synchronization and community detection,'' in {\em Conference on learning theory}, pp.~361--382, PMLR, 2016.

\bibitem{mei2017solving}
S.~Mei, T.~Misiakiewicz, A.~Montanari, and R.~I. Oliveira, ``Solving sdps for synchronization and maxcut problems via the grothendieck inequality,'' in {\em Conference on learning theory}, pp.~1476--1515, PMLR, 2017.

\bibitem{burer2003nonlinear}
S.~Burer and R.~D. Monteiro, ``A nonlinear programming algorithm for solving semidefinite programs via low-rank factorization,'' {\em Mathematical programming}, vol.~95, no.~2, pp.~329--357, 2003.

\bibitem{javanmard2016phase}
A.~Javanmard, A.~Montanari, and F.~Ricci-Tersenghi, ``Phase transitions in semidefinite relaxations,'' {\em Proceedings of the National Academy of Sciences}, vol.~113, no.~16, pp.~E2218--E2223, 2016.

\bibitem{wang2017mixing}
P.-W. Wang, W.-C. Chang, and J.~Z. Kolter, ``The mixing method: low-rank coordinate descent for semidefinite programming with diagonal constraints,'' {\em arXiv preprint arXiv:1706.00476}, 2017.

\bibitem{erdogdu2022convergence}
M.~A. Erdogdu, A.~Ozdaglar, P.~A. Parrilo, and N.~D. Vanli, ``Convergence rate of block-coordinate maximization burer--monteiro method for solving large sdps,'' {\em Mathematical Programming}, vol.~195, no.~1, pp.~243--281, 2022.

\bibitem{absil2007trust}
P.-A. Absil, C.~G. Baker, and K.~A. Gallivan, ``Trust-region methods on riemannian manifolds,'' {\em Foundations of Computational Mathematics}, vol.~7, pp.~303--330, 2007.

\bibitem{boumal2016non}
N.~Boumal, V.~Voroninski, and A.~Bandeira, ``The non-convex burer-monteiro approach works on smooth semidefinite programs,'' {\em Advances in Neural Information Processing Systems}, vol.~29, 2016.

\bibitem{horn1994topics}
R.~A. Horn and C.~R. Johnson, {\em Topics in matrix analysis}.
\newblock Cambridge university press, 1994.

\bibitem{burer2005local}
S.~Burer and R.~D. Monteiro, ``Local minima and convergence in low-rank semidefinite programming,'' {\em Mathematical programming}, vol.~103, no.~3, pp.~427--444, 2005.

\bibitem{boumal2014manopt}
N.~Boumal, B.~Mishra, P.-A. Absil, and R.~Sepulchre, ``Manopt, a matlab toolbox for optimization on manifolds,'' {\em The Journal of Machine Learning Research}, vol.~15, no.~1, pp.~1455--1459, 2014.

\bibitem{cupy_learningsys2017}
R.~Okuta, Y.~Unno, D.~Nishino, S.~Hido, and C.~Loomis, ``Cupy: A numpy-compatible library for nvidia gpu calculations,'' in {\em Proceedings of Workshop on Machine Learning Systems (LearningSys) in The Thirty-first Annual Conference on Neural Information Processing Systems (NIPS)}, 2017.

\bibitem{jax2018github}
J.~Bradbury, R.~Frostig, P.~Hawkins, M.~J. Johnson, C.~Leary, D.~Maclaurin, G.~Necula, A.~Paszke, J.~Vander{P}las, S.~Wanderman-{M}ilne, and Q.~Zhang, ``{JAX}: composable transformations of {P}ython+{N}um{P}y programs,'' 2018.

\bibitem{pymanopt2016}
J.~Townsend, N.~Koep, and S.~Weichwald, ``Pymanopt: A python toolbox for optimization on manifolds using automatic differentiation,'' {\em Journal of Machine Learning Research}, vol.~17, no.~137, p.~1–5, 2016.

\bibitem{maurer1979first}
H.~Maurer and J.~Zowe, ``First and second-order necessary and sufficient optimality conditions for infinite-dimensional programming problems,'' {\em Mathematical programming}, vol.~16, pp.~98--110, 1979.

\end{thebibliography}

\appendix

\section{Convergence to a Strict Local Minimum over Polyhedra}
\label{ap:local_min_polyhedra}
In this section, we show that by slightly changing Algorithm \ref{alg:gfw}, we can show convergence to a strict local minimum instead of a stationary point when $\setX$ is a polytope. 

Let us recall the following first-order sufficient condition for strict local minimality when the feasible set is a polytope. 
\begin{proposition}
\label{prop:polyhedra_local_min}
Let $g: \RR^d \mapsto \RR$ be a continuously differentiable function and $\setX \subset \RR^d$ be a nonempty polytope. Let $x^*$ be strictly stationary. In other words, $x^*$ satisfies

\begin{equation}
\label{eq:strict_stationary}
\begin{array}{lllll}
\nabla g(x^*) ^T(y-x^*) < 0, \quad \forall y \in \setX, y\neq x^*.
\end{array}
\end{equation}

Then $x^*$ is a strict local maximum of $g$ over $\setX$.
\end{proposition}

\begin{proof}
We will show that there exist $\beta >0$ and $r > 0$ such that $g(x_k) \leq g(x^*) - \beta \|x-x^*\|$ for all $x \in \setX$ satisfying $\|x-x^*\| \leq r$. Assume to the contrary that there exists a sequence $\{x_k\}_{k \in \NN} \subseteq \setX$ converging to $x^*$ such that $g(x_k) > g(x^*) - \frac{1}{k}\|x-x^*\|$ for all $k$. Then, define the following vector

\begin{equation}
\label{eq:feasible_direction}
\begin{array}{lllll}
p_k = \frac{x_k-x^*}{\|x_k-x^*\|_2}
\end{array}
\end{equation}

$p_k$ is feasible direction at $x^*$. The sequence $\{p_k\}_{k\in \NN}$ lies on a compact set and, therefore, there exists a convergent subsequence. For convenience, assume that the sequence $p_k$ converges to some nonzero $\Bar{p}$ without loss of generality. Since $\setX$ is a polytope, $\Bar{p}$ is also a feasible direction. Then by Taylor's Theorem, we have,

\begin{equation}
\label{eq:feasible_direction_taylor}
\begin{array}{rllll}
g(x_k) &= g(x^*) + \nabla g(x^*)^T(x_k-x^*) + O(\|x_k-x^*\|) \\
-\frac{1}{k} &< \nabla g(x^*)^T( \frac{x_k-x^*}{\|x_k-x^*\|}) + \frac{O(\|x_k-x^*\|)}{\|x_k-x^*\|} \\
0 &\leq \nabla g(x^*)^T\Bar{p} 
\end{array}
\end{equation}
In the first line, we used a first-order approximation. In the second line, we used the assumption that $g(x_k) > g(x^*) - \frac{1}{k}\|x-x^*\|$ and divided both sides of the equation by $\|x_k-x^*\|$. In the last line, we let $k$ go to infinity. The final equation creates a contradiction to (\ref{eq:strict_stationary}). In conclusion, $x^*$ must be a strict local maximum. 
\end{proof} 

A substantially more general version of the Proposition \ref{prop:polyhedra_local_min} was proved in \cite{maurer1979first}. Remark that (\ref{eq:strict_stationary}) is also the second-order necessary condition when $g$ is strongly convex and $\setX$ is convex. To the contrary, if there exists a $y \in \setX, y\neq x$ such that $g(x^*) ^T(y-x^*) = 0$. Then, $\{x^*+ \frac{1}{n}y\}$ is a sequence of points in $\setX$ that converges to $x^*$ with a higher objective. Additionally, even if $\setX$ is not convex, although $x$ might be a local minimum, point $y$ still improves upon $x$ in terms of objective value. This motivates checking alternative solutions when applying the maximization step in the GFW algorithm. In particular, while working over a polytope, if we can check for alternative optimal solutions, we can enforce the strict stationarity condition (\ref{eq:strict_stationary}).  Finally, this augmentation forces the GFW algorithm to converge to a strict local maximum instead of a stationary point by Proposition \ref{prop:polyhedra_local_min} in a finite number of steps since the objective is increasing to a limit and there are finitely many extreme points of $\setX$.

\end{document}